\documentclass[a4paper, 11pt]{article}
\usepackage{amsfonts}
\usepackage {amssymb}
\usepackage {amsmath}
\usepackage {amsmath}
\usepackage {amsthm}
\usepackage{graphicx}
\usepackage {amscd}
\usepackage[colorlinks, linkcolor=blue, anchorcolor=black, citecolor=red]{hyperref}
\usepackage{enumerate}

\usepackage{geometry}  

\newcommand{\Fut}{{\rm Fut}}

\newcommand{\sddb}{{\sqrt{-1}\partial\bar{\partial}}}

\newcommand{\bL}{{\bf L}}

\newcommand{\vol}{{\rm vol}}
\newcommand{\ord}{{\rm ord}}
\newcommand{\lct}{{\rm lct}}
\newcommand{\vphi}{\varphi}

\newcommand{\bC}{{\mathbb{C}}}

\newcommand{\tr}{{\rm tr}}

\newcommand{\bP}{{\mathbb{P}}}

\newcommand{\cE}{{\mathcal{E}}}
\newcommand{\NA}{{\rm NA}}

\newcommand{\cJ}{{\mathcal{J}}}
\newcommand{\cX}{{\mathcal{X}}}
\newcommand{\cL}{{\mathcal{L}}}

\newcommand{\bQ}{{\mathbb{Q}}}
\newcommand{\reg}{{\rm reg}}

\newcommand{\bR}{{\mathbb{R}}}
\newcommand{\CM}{{\rm CM}}

\newcommand{\cH}{{\mathcal{H}}}
\newcommand{\cO}{{\mathcal{O}}}
\newcommand{\cF}{{\mathcal{F}}}
\newcommand{\bZ}{{\mathbb{Z}}}

\newcommand{\bD}{{\mathbb{D}}}

\newcommand{\sing}{{\rm sing}}

\newcommand{\bN}{{\mathbb{N}}}

\newcommand{\la}{\langle}
\newcommand{\ra}{\rangle}
\newcommand{\cY}{\mathcal{Y}}
\newcommand{\bfL}{{\bf L}}

\newcommand{\bI}{{\bf I}}

\newcommand{\mes}{{\rm mes}}
\newcommand{\bfI}{{\bf I}}

\newcommand{\bd}{{\rm bd}}

\newcommand{\mcZ}{\mathcal{Z}}
\newcommand{\mcL}{\mathcal{L}}
\newcommand{\mcQ}{\mathcal{Q}}
\newcommand{\mcY}{\mathcal{Y}}
\newcommand{\bfE}{{\bf E}}
\newcommand{\bfJ}{{\bf J}}
\newcommand{\bfD}{{\bf D}}

\newcommand{\bfF}{{\bf F}}
\newcommand{\mcH}{\mathcal{H}}
\newcommand{\mcE}{\mathcal{E}}
\newcommand{\bfH}{{\bf H}}
\newcommand{\bfM}{{\bf M}}

\newcommand{\Aut}{{\rm Aut}}

\newcommand{\mcI}{\mathcal{I}}

\newcommand{\PSH}{{\rm PSH}}

\newcommand{\mcJ}{\mathcal{J}}
\newcommand{\mcO}{\mathcal{O}}
\newcommand{\bfT}{{\bf T}}

\newtheorem{thm}{Theorem}[section]
\newtheorem{prop}[thm]{Proposition}
\newtheorem{defn}[thm]{Definition}

\newtheorem{cor}[thm]{Corollary}
\newtheorem{rem}[thm]{Remark}

\newtheorem{exmp}[thm]{Example}
\newtheorem{lem}[thm]{Lemma}

\begin{document}

\title{The uniform version of Yau-Tian-Donaldson conjecture for singular Fano varieties}
\author{Chi Li, Gang Tian, Feng Wang}
\date{}

\maketitle

\abstract{
We prove the following result: if a $\bQ$-Fano variety is uniformly K-stable, then it admits a K\"{a}hler-Einstein metric. This proves the uniform version of Yau-Tian-Donaldson conjecture for all (singular) Fano varieties with discrete automorphism groups.
We achieve this by modifying Berman-Boucksom-Jonsson's strategy in the smooth case with appropriate perturbative arguments. This perturbation approach depends on the valuative criterion and non-Archimedean estimates, and is motivated by our previous paper.
}

\tableofcontents

\section{Introduction}

A Fano variety is defined to be a normal projective variety $X$ such that its anticanonical divisor $-K_X$ is an ample $\bQ$-Cartier divisor. K-(poly)stability of Fano varieties was introduced by Tian in \cite{Tia97} and reformulated more algebraically by Donaldson \cite{Don02}. The Yau-Tian-Donaldson (YTD) conjecture states that a smooth Fano manifold $X$ admits a K\"{a}hler-Einstein metric if and only if $X$ is K-polystable. 
Due to many people's work, this conjecture has been proved (see \cite{Tia97, Berm15, CDS15, Tia15}). 

In this paper, we are interested in the generalized Yau-Tian-Donaldson conjecture meaning that $X$ is allowed to be singular. There are some previous works \cite{SSY16, LWX14} and \cite{LTW17} on extending the YTD conjecture to special classes of singular Fano varieties.
Berman's work in \cite{Berm15} shows that the ``only if" part of the conjecture is indeed true for any log Fano pair. For the ``if" part,  Note that, by \cite{Oda13}, a Fano variety $X$ being K-semistable implies that $X$ has at worst klt singularities (see also \cite{LX14}). Fano varieties with at worst klt singularities will be called $\bQ$-Fano varieties.

When $X$ has a discrete automorphism group, K-polystability is also called K-stability. In this case, the notion of uniform K-stability as defined in \cite{BHJ17, Der16} is the algebraic correspondent to the properness of Mabuchi energy and is a prior a strengthening of the K-stability condition. 
The uniform K-stability is actually conjectured to be equivalent to K-stability. This is known in the smooth case through the solution of Yau-Tian-Donaldson conjecture.
Uniform K-stability has recently been studied extensively. For example Fujita (\cite{Fuj17a}) proved that there is a nice valuative criterion for uniform K-stability (see also \cite{BlJ17, BoJ18b, FO16}), and moreover, uniformly K-stable Fano varieties with a fixed volume are parametrized by a good moduli stack (see \cite{BX18}).

In this paper, we will in fact deal with the more general case of log-Fano pairs $(X, D)$ (see Definition \ref{defn-logFano}) and prove the following main result:
\begin{thm}\label{thm-main}
Assume that a log-Fano pair $(X, D)$ is uniformly K-stable. Then the Mabuchi energy of $(X, D)$ is proper over the space of finite energy $\omega$-psh potentials for any fixed reference smooth K\"{a}hler metric $\omega$. 
\end{thm}
By the work of \cite{BBEGZ, DR17, Dar17, DNG18}, if $(X, D)$ has a discrete automorphism group, then $(X, D)$ has a K\"{a}hler-Einstein metric if and only if the Mabuchi energy is proper over the space of finite energy K\"{a}hler metrics (see Theorem \ref{thm-analytic}).
Moreover, the latter condition indeed implies that $(X, D)$ is uniformly K-stable (see \cite{Berm15, BHJ19}). So we get the following version of Yau-Tian-Donaldson conjecture.
\begin{cor}
Assume that a log-Fano pair $(X, D)$ has a discrete automorphism group. Then $(X, D)$ has a K\"{a}hler-Einstein metric if and only if $(X, D)$ is uniformly K-stable.
\end{cor}
We emphasize that the pair $(X, D)$ in the above result is allowed to have any klt singularities.
As mentioned above, by the resolution of YTD conjecture in the smooth case, the above results are known when $X$ is smooth and $D$ is zero.

The above theorem extends the work of Berman-Boucksom-Jonsson in \cite{BBJ15,BBJ18} to the class of singular $\bQ$-Fano varieties. 
Indeed our method of proof will be based on the strategy proposed by Berman-Boucksom-Jonsson in \cite{BBJ15}. In particular, we depend on various tools from pluripotential theory, non-Archimedean K\"{a}hler geometry and birational algebraic geometry, but without using Cheeger-Colding-Tian's theory and partial $C^0$-estimates in the original solution of the YTD conjecture.
However, as explained in \cite{BBJ18}, there are technical difficulties in applying their method directly when $X$ is singular (see section \ref{sec-BBJ}). Here we use some ideas of perturbative approach. 
 
 It is well known that to solve K\"{a}hler-Einstein metrics on singular varieties is equivalent to solve some degenerate Monge-Amp\`{e}re equation on a resolution of the variety (see \cite{EGZ09,BBEGZ}). It is natural to study such degenerate Monge-Am\`{e}re equation using an appropriate sequence of non-degenerate Monge-Amp\`{e}re equations to approximate the original equation, which is the guiding principle in \cite{LTW17}. The perturbative approach used here is motivated by this idea. In fact we need to apply the appropriate perturbative method in every step of our argument. 

In the next section, we will first discuss Berman-Boucksom-Jonsson's variational approach to YTD and our previous work in \cite{LTW17} which uses perturbation arguments to prove YTD for a class of singular Fano varieties. These will serve as comparisons to our new argument to get the uniform version of YTD in the singular case, which we sketch in section \ref{sec-skmod} highlighting some new ingredients about convergence of non-Archimedean quantities. In section \ref{sec-pre}, we recall the preliminary materials on space of K\"{a}hler metrics on singular varieties which were developed in \cite{BBEGZ, Dar17, DNG18}. We state 
the analytic criterion for the existence of KE metrics on singular Fano varieties as studied by Darvas and Di-Nezza-Guedj, and slightly refine it by using the approximation argument of Berman-Darvas-Lu. The reason for doing this is the observation that the argument to get KE on singular Fano varieties would be easier if we know that the properness over Mabuchi energy over the space of smooth K\"{a}hler potentials implies the existence of KE (see Remark \ref{rem-smoothproper} and \ref{rem-smoothconvex}). Since the latter is not known, we need to work more in section \ref{sec-convex} to get the convexity of Mabuchi energy along geodesic segments connecting less regular positively curved Hermitian metrics.
In section \ref{sec-TC}, we recall the definitions of K-stability and its equivalent Ding stability. In section \ref{sec-valuative}, we will recall  the non-Archimedean formulation of stability conditions and the valuative criterion for uniform stability. We also observe that the valuative criterion still works when the boundary divisor is non-effective at least when the ambient space is smooth. In section \ref{sec-proof}, we prove our main results by following the steps as sketched in section \ref{sec-skmod}.

The results of this paper have been generalized to the general case when $\Aut(X)_0$ is reductive but not necessarily discrete (\cite{Li19}, see also \cite{His19}), and to more general equations of twisted KE's and generalized solitons (see \cite{HL20}). 

\vskip 3mm
\noindent
{\bf Acknowledgement:} 
C. Li is partially supported by NSF (Grant No. DMS-1810867) and an Alfred P. Sloan research fellowship. 
G. Tian is partially supported by
NSF (Grant No. DMS-1607091) and NSFC (Grant No. 11331001). F. Wang is partially supported by NSFC (Grant No.11501501).
The first author would like to thank S. Boucksom, M. Jonsson and L. Lempert for helpful conversations, and Y. Liu, C. Xu and M. Xia for useful comments. We would like to thank R. Berman, T. Darvas for communications that help our proof of the convexity of Mabuchi energy, and Di Nezza and V. Guedj for clarifications on regularity of geodesics. We would also like to thank anonymous referees for helpful suggestions on improving the paper.

\section{Discussion of proofs}

In this section, we first sketch and discuss the variational approach of Berman-Boucksom-Jonsson (BBJ) \cite{BBJ15, BBJ18} and the perturbative approach of Li-Tian-Wang (\cite{LTW17}). Then we sketch our proof which is a modification of BBJ's approach by instilling some perturbative idea and convergence results. In the following sketch we will only consider the case when the boundary divisor is empty. We will also use the equivalence of uniform K-stability and uniform Ding stability for any $\bQ$-Fano variety as proved in \cite{BBJ15, Fuj16}. See section \ref{sec-pre} for the notations used in the following sketch.

\subsection{Berman-Boucksom-Jonsson's approach}\label{sec-BBJ}
We first sketch Berman-Boucksom-Jonsson's proof of the smooth case of Theorem \ref{thm-main}. Assume a smooth Fano manifold $X$ is uniformly K-stable. They proved the properness of Mabuchi energy using a proof by contradiction. 
\begin{enumerate}
\item Step 1: Assume on the contrary that the Mabuchi energy $\bfM$ is not proper, then one can find a destabilizing geodesic ray $\Phi=\{\vphi(s)\}_{s\in [0, \infty)}$ in $\cE^1:=\cE^1(X, -K_X)$ such that 
\begin{enumerate}
\item
$\bfM$ and the Ding energy $\bfD$ are decreasing along $\Phi$. In particular, we have
\begin{equation}\label{eq-Dinfneg1}
\bfD'^\infty(\Phi):=\lim_{s\rightarrow+\infty} \frac{\bfD(\vphi(s))}{s}\le 0.
\end{equation}
\item With a smooth reference Hermitian metric $\psi_0\in \mcE^1$, we have the following normalization
\begin{equation}
\sup(\vphi(s)-\psi_0)=0, \quad  \bfE_{\psi_0}(\vphi(s))=- s.
\end{equation}
\end{enumerate}
\item Step 2: For $m\gg 1$, blow up the multiplier ideal sheaf $\cJ(m\Phi)$ to construct a sequence of semi-ample test configurations $\cX_m$ whose associated psh-ray and non-Archimedean metric will be denoted by $\Phi_m=\{\vphi_m(s)\}$ and $\Phi_m^\NA$.

Moreover, Demailly's regularization theorem (\cite[Proposition 3.1]{Dem92}) implies that $\Phi_m$ is less singular than $\Phi$. This together with the monotonicity of the $\bfE$ functional show that 
\begin{equation}\label{eq-sk1}
\bfE'^\infty(\Phi_m)=\lim_{s\rightarrow+\infty} \frac{\bfE(\vphi_m(s))}{s} \ge \lim_{s\rightarrow+\infty} \frac{\bfE(\vphi(s))}{s}=:\bfE'^\infty(\Phi)=-1.
\end{equation}
As noted in \cite[Corollary 6.7]{BBJ18}, this may a priori be a strict inequality without knowing that $\Phi$ is a maximal geodesic ray. 

\item Step 3: Prove the following expansion of $\bL$ energy along $\Phi$ by generalizing \cite{Berm15} and using the valuative tools from \cite{BFJ08}:
\begin{equation}
\lim_{s\rightarrow+\infty} \frac{\bL(\vphi(s))}{s}=\inf_{w\in W} (A_{X\times\bC}(w)-w(\Phi))-1=:\bL^\NA(\Phi^\NA),
\end{equation} 
where $W$ is the set of $\bC^*$-invariant divisorial valuations $w$ on $X\times\bC$ with $w(t)=1$.

Moreover, use Demailly's regularization result and definition of multiplier ideals to prove that:
\begin{equation}\label{eq-sk2}
\lim_{m\rightarrow+\infty} \bL^\NA(\Phi^\NA_m)=\bL^\NA(\Phi^\NA).
\end{equation}
\item Step 4: Combine \eqref{eq-Dinfneg1}-\eqref{eq-sk2} to prove that $\Phi$ contradicts the uniform Ding-stability of $X$, which is equivalent to the uniform K-stability.

\end{enumerate}

As pointed out in \cite{BBJ18}, a large part of the above arguments in \cite{BBJ18} still applies to singular $\bQ$-Fano varieties. 
The difficulty in the singular case lies essentially in applying Demailly's regularization directly on singular varieties. This regularization result is in general not true when the ambient space is singular and Berman-Boucksom-Jonsson suggested to find a replacement of this regularization result for singular varieties. The other difficulty may lie in the study of non-Archimedean spaces over singular varieties. Our main contribution will be to circumvent these difficulties.

\subsection{Perturbation approach of Li-Tian-Wang}
Theorem \ref{thm-main} has been proved in a special singular case in \cite{LTW17}, which we will recall in this subsection. Let $X$ be any $\bQ$-Fano variety. Take a log resolution $\mu: Y\rightarrow X$ such that the reduced exceptional divisor $\mu^{-1}(X^\sing)=\sum_{k=1}^g E_k$ is a simple normal crossing divisor. The klt condition allows one to write down the following identity:
\begin{equation}\label{eq-KY/X}
K_Y=\mu^*K_X+\sum_{k=1}^g a_k E_k=\mu^*K_X-\sum_{i=1}^{g_1} b_i E'_i+\sum_{j=g_1+1}^g a_j E''_j,
\end{equation}
where for $i=1,\dots, g_1$, $E'_i=E_i$, $b_i=-a_i\in [0,1)$; and for $j=g_1+1,\dots, g$, $a_j>0$ and $E''_j=E_j$.

It is well known (e.g. \cite[Lemma 2.2]{CMM17}) that we may and will assume that there exists a log resolution $\mu: Y\rightarrow X$ such that for some $\theta_k\in \bQ$ with $0<\theta\ll 1$, $k=1,\dots, g$
\begin{equation}\label{eq-Pep}
P:=\mu^*(-K_X)-\sum_{k=1}^g \theta_k E_k \quad \text{ is positive.}
\end{equation} 
We can then rewrite the identity \eqref{eq-KY/X} in the following way:
\begin{eqnarray}\label{eq-dec-K}
-K_Y
&=&\frac{1}{1+\epsilon}\left((1+\epsilon)\mu^*(-K_X)-\epsilon \sum_k \theta_k E_k\right)\nonumber \\
&&\hskip 3cm +\sum_i (b_i+\frac{\epsilon}{1+\epsilon}\theta_i)E'_i
-\sum_j (a_j-\frac{\epsilon}{1+\epsilon}\theta_j)E''_j\nonumber \\
&=&\frac{1}{1+\epsilon}L_\epsilon+B_\epsilon,
\end{eqnarray}
where for simplicity we introduced the following notations for any $\epsilon\ge 0$:
\begin{eqnarray}\label{eq-short}
L_\epsilon&:=&(1+\epsilon)\mu^*(-K_X)-\epsilon\sum_{k=1}^g \theta_k E_k=\mu^*(-K_X)+\epsilon P;  \nonumber \\
B^+_\epsilon&:=&\sum_{i=1}^{g_1} (b_i+\frac{\epsilon}{1+\epsilon}\theta_i)E'_i, \quad B^-_\epsilon:=\sum_{j=g_1+1}^g (a_j-\frac{\epsilon}{1+\epsilon}\theta_j)E''_j \nonumber \\
B_\epsilon&:=&B^+_\epsilon-B^-_\epsilon.
\end{eqnarray}
Note that $B^-_0=\sum_{j=g_1+1}^g a_j E_j=0$ if and only if $-1<a_k\le 0$ for any $k=1,\dots, g$. 
One intermediate result in \cite{LTW17} can be stated as follows:
\begin{thm}[{\cite[Theorem 4.11]{LTW17}}]\label{thm-LTW1}
Let $X$ be a $\bQ$-Fano variety. Assume that there is a log resolution $\mu: Y\rightarrow X$ satisfying both \eqref{eq-Pep} and $B^-_0=0$. If $X$ is uniformly K-stable, then there exists a K\"{a}hler-Einstein metric on $X$.
\end{thm}
Let's very briefly sketch the proof of the above result:
\begin{enumerate}
\item Prove that $(Y, B_\epsilon)$ is uniformly K-stable or equivalently uniformly Ding-stable. Note that by assumption $B_\epsilon=B^+_\epsilon\ge 0$ and $(Y, B_\epsilon)$ is a klt pair.
This is achieved by using the valuative criterion of uniform K-stability by Fujita.
\item Adapt Berman-Boucksom-Jonsson's result to the logarithmic setting to prove that the Mabuchi energy of $(Y, B_\epsilon)$ is proper with slope constants that are uniform with respect to $\epsilon$.
This in particular implies that there exists a K\"{a}hler-Einstein metric $\omega_\epsilon:=\sddb \vphi_\epsilon$ with edge cone singularities (see \cite{JMR16, GP16}) on the klt pair $(Y, B_\epsilon)$ where $e^{-\vphi_\epsilon}$ is an Hermitian metric on $L_\epsilon=-(1+\epsilon)(K_Y+B_\epsilon)$.
\item Prove the convergence of $\vphi_\epsilon$ as $\epsilon\rightarrow 0^+$ by proving uniform estimates by comparing energy functionals on $X$ and $(Y, B_\epsilon)$ with some rescaling argument and using uniform Sobolev constants of K\"{a}hler-Einstein metrics with edge cone singularities.
\end{enumerate}
There are also serious difficulties in this perturbative approach for the general singular case. Indeed if $B^-_\epsilon>0$, then $B_\epsilon$ is not effective. Any K\"{a}hler-Einstein metric on the {\it ineffective pair} $(Y, B^+_\epsilon-B^-_\epsilon)$, if it exists, would have edge cone singularities of cone angles bigger than $2\pi$ along ${\rm supp}(B_\epsilon^-)$. It is still not clear how to adapt Berman-Boucksom-Jonsson's variational approach to construct such a singular K\"{a}hler-Einstein metric. 
For example, when $\epsilon>0$, the Mabuchi energy is not known to be convex along geodesics (since the twisting is non-effective). Even if one could do so, several analytic and geometric tools in our original arguments are missing for such singular K\"{a}hler metrics. However, we should point out that the advantage of this perturbative approach is that it allows us to further
combine Cheeger-Colding-Tian's theory (extended in the edge cone situation by Tian-F. Wang) and partial $C^0$-estimates for conical K\"{a}hler-Einstein metrics to get a full (K-polystable) version of Yau-Tian-Donaldson conjecture in the special singular class.

\subsection{Perturbing BBJ's argument}\label{sec-skmod}

We now sketch the argument in our proof of Theorem \ref{thm-main}. We will use the above notations (and notations from section \ref{sec-pre}) and prove by contradiction.
So assume that the $\bQ$-Fano variety $X$ is uniformly K-stable. 
\begin{enumerate}
\item Step 1:Assume on the contrary that the Mabuchi energy $\bfM$ is not proper. We first prove by perturbative approach that the Mabuchi energy is convex along appropriate geodesic segments. Then we find a destabilizing geodesic ray $\Phi=\{\vphi(s)\}_{s\in [0, \infty)}$ in $\cE^1:=\cE^1(X, -K_X)$ such that 
\begin{enumerate}
\item
$\bfM$ and the Ding energy $\bfD$ are decreasing along $\Phi$. In particular, we have
\begin{equation}\label{eq-Dinfneg}
\bfD'^\infty(\Phi):=\lim_{s\rightarrow+\infty} \frac{\bfD(\vphi(s))}{s}\le 0.
\end{equation}
\item With a smooth Hermitian metric $\psi_0\in \mcE^1$, we have the following normalization
\begin{equation}
\sup(\vphi(s)-\psi_0)=0, \quad  \bfE_{\psi_0}(\vphi(s))=- s.
\end{equation}
\end{enumerate}

\item Step 2: 
Fix a log resolution $\mu: Y\rightarrow X$ satisfying \eqref{eq-Pep}.
Consider the psh ray on $L_\epsilon=\mu^*(-K_X)+\epsilon P$ given by $$\Phi_\epsilon=\mu^*\Phi+\epsilon p'^*_1\psi_P$$
where $\psi_P$ is a smooth Hermitian metric on $P=\mu^*(-K_X)-\sum_{k=1}^g \theta_k E_k$ whose curvature is a smooth K\"{a}hler form, and $p'_1: Y\times\bC\rightarrow Y$ is the projection.
Blow up the multiplier ideal sheaf $\cJ(m \Phi_\epsilon)$ to construct test configurations $(\cY_{\epsilon,m},\cL_{\epsilon,m})$ of $(Y, L_\epsilon)$ whose associated psh-ray and non-Archimedan metric are denoted by $\Phi_{\epsilon,m}$ and $\Phi^\NA_{\epsilon,m}$.

Demailly's regularization result on $Y$ implies that (see \eqref{eq-ENAepmlb}):
\begin{eqnarray}\label{eq-sk1'}
\bfE'^\infty_{L_\epsilon}(\Phi_{\epsilon,m})=\lim_{s\rightarrow+\infty} \frac{\bfE(\vphi_{\epsilon,m}(s))}{s} \ge \lim_{s\rightarrow+\infty}\frac{\bfE(\vphi_{\epsilon}(s))}{s}=:\bfE'^\infty_{L_\epsilon}(\Phi_{\epsilon}).
\end{eqnarray}
Moreover, we prove the following convergence (see \eqref{eq-limENAep}):
\begin{equation}\label{eq-sklimEep}
\lim_{\epsilon\rightarrow 0}\bfE'^\infty_{L_\epsilon}(\Phi_{\epsilon})=\bfE'^\infty(\Phi)=-1.
\end{equation}

\item Step 3: Prove an expansion of $\bL_{(Y, B_\epsilon)}$ along any psh ray on $(Y, L_\epsilon)$ by adapting the proof in \cite{BBJ15, BBJ18} (see Proposition \ref{prop-LBexpan}):
\begin{equation}
\lim_{s\rightarrow+\infty} \frac{\bL_{(Y, B_\epsilon)}(\vphi_\epsilon(s))}{s}=\bL^\NA_{(Y, B_\epsilon)}(\Phi^\NA_\epsilon).
\end{equation}

Use Demailly's regularization on $Y$ to prove (see \eqref{eq-limLNAm}):
\begin{equation}\label{eq-sk2'}
\lim_{m\rightarrow+\infty} \bL^\NA_{(Y, B_\epsilon)}(\Phi^\NA_{\epsilon,m})=\bL^\NA_{(Y, B_\epsilon)}(\Phi^\NA_\epsilon).
\end{equation}

Moreover we prove the following convergence (see \eqref{eq-limLNAep}):
\begin{equation}\label{eq-sklimLep}
\lim_{\epsilon\rightarrow 0} \bL^\NA_{(Y, B_\epsilon)}(\Phi^\NA_\epsilon)=\bL^\NA(\Phi^\NA).
\end{equation}

\item Step 4: Prove that the uniform K-stability of $X$ implies the uniform Ding-stability of $(Y, B_\epsilon)$ for $0<\epsilon\ll 1$ where $B_\epsilon$ is the not-necessarily effective $\bQ$-divisor in \eqref{eq-short}.

\item Step 5: Combine \eqref{eq-sk1'}-\eqref{eq-sklimLep} to prove that $\Phi_\epsilon$ contradicts the uniform Ding-stability of $(Y, B_\epsilon)$ for $0<\epsilon\ll 1$.

\end{enumerate}

Although the general strategy is in Berman-Boucksom-Jonsson's framework, the details are a lot more technical. On the other hand, we will only use Demailly's regularization on the smooth $Y$. Moreover, the perturbative part is indispensable. In particular, the convergences in the new arguments \eqref{eq-sklimEep} and \eqref{eq-sklimLep} are crucial, and the Step 4 is directly analogous to the first step in \cite{LTW17}.
The idea of using perturbative approach here is suggested by our previous work in \cite{LTW17}. However, instead of working with the energy functional on the space of K\"{a}hler metrics as in \cite{LTW17}, we will be working more on the non-Archimedean side, which is more flexible in some sense due to the birational-nature of the valuative criterions developed in \cite{BoJ18b, Fuj16}. Equally important in our arguments is the observation that some of the non-Archimedean arguments in \cite{BoJ18b, BBJ18} work well for the non-effective twisting at hand.

\section{Preliminaries}\label{sec-pre}

\subsection{Space of K\"{a}hler metrics over singular varieties}
Let $Z$ be an $n$-dimensional normal projective variety and $Q$ a Weil divisor that is not necessarily effective. Assume that $L$ is an ample $\bQ$-Cartier divisor. Choose a smooth Hermitian metric $e^{-\psi}$ on $L$ with a smooth semi-positive curvature form $\omega=\sddb\psi\in 2\pi c_1(L)$. 

Recall that  a function $u: Z\rightarrow [-\infty, +\infty)$ is $\omega$-plurisubharmonic ($\omega$-psh for short) if $u+\psi$  is a plurisubharmonic function for each local potential $\psi$ of $\omega$ (see \cite[1.1]{BBEGZ}). 
We will use the following spaces:
\begin{align}
&{\rm PSH}(\omega):=\PSH(Z, \omega)=\left\{u: Z\rightarrow [-\infty, +\infty); u \text{ is $\omega$-psh} \right\};\\
&\mcH(\omega):=\mcH(Z, \omega)={\rm PSH}(\omega)\cap C^\infty(Z);\\
&{\rm PSH}_\bd(\omega):=\PSH_\bd(Z, \omega)=\PSH(\omega)\cap \{\text{bounded functions on } Z\};\\
&{\rm PSH}(L):={\rm PSH}([\omega]):=\left\{\vphi=\psi+u; u\in {\rm PSH}(\omega)\right\};\\
&{\rm PSH}_\bd(L):={\rm PSH}_\bd([\omega]):=\left\{\vphi=\psi+u; u\in {\rm PSH}_\bd(\omega)\right\}.
\end{align}
Note that ${\rm PSH}([\omega])$ is equal to the space of positively curved (possibly singular) Hermitian metrics $\{e^{-\vphi}=e^{-\psi-u}\}$ on the $\bQ$-line bundle $L$. Rigorously $\psi+u$ is not a globally defined function, but rather a collection of local psh functions that satisfy the obvious compatible condition with respect to the transition functions of the $\bQ$-line bundle. However for the simplicity of notations, we will abuse this notation.

We have weak topology on ${\rm PSH}(\omega)$ which coincides with the $L^1_{loc}$-topology with respect to the smooth volume form $\omega^n$. 
If $u_j$ converges to $u$ weakly, then $\sup(u_j)\rightarrow \sup(u)$ by Hartogs' lemma for plurisubharmonic functions (see \cite[1.4]{GZ16}).
\begin{prop}[{\cite[Corollary C]{CGZ13}}]\label{prop-smapp}
For any $u\in \PSH(Z, \omega)$ there exists a sequence of smooth functions $u_j\in \PSH(Z, \omega)$ which decrease pointwise on $Z$ so that $\lim_{j\rightarrow+\infty}u_j=u$ on $Z$.
\end{prop}
For any $u\in \PSH(Z, \omega)$, define:
\begin{equation}
\omega_u^n:=\lim_{j\rightarrow+\infty}  {\bf 1}_{\{u>-j\}}\left(\omega+\sddb \max(u, -j)\right)^n.
\end{equation}
We will use the space $\mcE^1$ of finite energy $\omega$-psh functions (see \cite{GZ07}):
\begin{align}
&\mcE(\omega):=\mcE(Z, \omega)=\left\{u\in {\rm PSH}(Z, \omega); \int_Z \omega_u^n=\int_Z \omega^n\right\};\\
&\mcE^1(\omega):=\mcE^1(Z,\omega)=\left\{u\in \mcE(Z, \omega); \int_Z |u| \omega_u^n<\infty\right\};\\
&\mcE^1(L):=\mcE^1(Z, L)=\left\{\psi+u; u \in \mcE^1(Z, \omega) \right\}.
\end{align}
We have the inclusion $\PSH_\bd(\omega)\subset \mcE^1(\omega)$.

For any $\vphi\in {\rm PSH}([\omega])$ such that $\vphi-\psi\in \mcE^1(L)$, we have the following important functional: 
\begin{eqnarray}
\bfE(\vphi)&:=&\bfE_{\psi}(\vphi)=\frac{1}{n+1} \sum_{i=0}^n \int_Z (\vphi-\psi) (\sddb\psi)^{n-i}\wedge (\sddb\vphi)^{i}.\label{eq-Ephi}
\end{eqnarray}
Following \cite{BBEGZ}, we endow $\mcE^1$ with the strong topology. 
\begin{defn}
The strong topology on $\mcE^1$ is defined to as the coarsest refinement of the weak topology such that $\bfE$ is continuous.
\end{defn}
We will use the following monotone and rescaling property of 
$\bfE$ functional:
\begin{eqnarray}\label{eq-monotoneE}
\vphi_1\le \vphi_2 &\Longrightarrow& 
 \bfE(\vphi_1)\le \bfE(\vphi_2); \quad \bfE_{\lambda \psi}(\lambda \vphi)=\lambda^{n+1} \cdot \bfE_{\psi}(\vphi) \text{ for any } \lambda\in \bR_{>0}. \label{eq-rescaleE}
\end{eqnarray}

For any interval $I\subset \bR$, denote the Riemann surface $$\bD_I=I\times S^1=\{\tau\in \bC^*; s=\log|\tau|\in I\}.$$ 
\begin{defn}[{see \cite[Definition 1.3]{BBJ18}}]\label{defn-pshpath}
A $\omega$-psh path, or just the psh path, on an open interval $I$ is a map $U=\{u(s)\}: I\rightarrow \PSH(\omega)$ such that the $U(\cdot, \tau):=U(\log|\tau|)$ is a $p_1^*\omega$-psh function on $X\times \bD_I$. A psh ray (emanating from $u_0$) is a psh path on $(0, +\infty)$ (with $\lim_{t\rightarrow 0}u(s)=u_0$). Note in the literature, psh path (resp. psh ray) are also called subgeodesic (resp. subgeodesic ray). 

In the above situation, we also say that $\Phi(s)=\{\psi_0+u(s)\}$ is a psh path (resp. a psh ray). 
\end{defn}
We will use geodesics connecting bounded potentials. 
\begin{prop}[{\cite[Proposition 1.17]{DNG18}}]\label{prop-geod}
Let $u_0, u_1\in \PSH_\bd(\omega)$. Then  
\begin{equation}\label{eq-envelope}
U=\sup\left\{u; u\in \PSH(Z\times\bD_{[0,1]}, p_1^*\omega);\quad U\le u_{0,1} \text{ on } \partial (Z\times \bD_{[0,1]})\right\}.
\end{equation}
is the unique bounded $\omega$-psh function on $Z\times\bD_{[0,1]}$ that is the solution of the Dirichlet problem:
\begin{equation}
(\omega+\sddb U)^{n+1}=0 \text{ on } Z\times\bD_{[0,1]}, \quad U|_{Z\times\partial \bD_{[0,1]}}=u_{0,1}.
\end{equation}
\end{prop}
We will call $\Phi=\{\vphi(s)=\psi+U(\cdot, s)\}$ the geodesic segment joining $\vphi_0=\psi+u_0$ and $\vphi_1=\psi+u_1$.

For finite energy potentials $u_0, u_1\in \mcE^1(\omega)$, let $u^j_0, u^j_1$ be bounded smooth $\omega$-psh functions decreasing to $u_0, u_1$ (see Proposition \ref{prop-smapp}). Let $u_t^j$ be the bounded geodesic connecting $u^j_0$ to $u^j_1$. It follows from the maximum principle that $j\rightarrow u^j_{t}$ is non-increasing. Set:
\begin{equation}
u_t:=\lim_{j\rightarrow+\infty}u^j_t.
\end{equation}
Then $U=\{u_t\}$ is a finite-energy geodesic joining $u_0$ to $u_1$ as stated in the following result.
\begin{thm}[{\cite[Proposition 4.6]{DNG18}, \cite[Theorem 1.7]{BBJ18}}]
For any $u_0, u_1\in \mcE^1(\omega)$, the psh geodesic joining them exists, and defines a continuous map $U: [0,1]\rightarrow \mcE^1$ in the strong topology.
\end{thm}

Generalizing Darvas' result in the smooth case (\cite{Dar15}), the works in \cite{Dar17, DNG18} showed that 
$\mcE^1$ can be characterized as the metric completion of $\mcH(\omega)$ under a Finsler metric $d_1$ which can be defined as follows.  
Fix a log resolution $\mu: Y\rightarrow Z$ and a K\"{a}hler form $\omega_P>0$ on $Y$. Then
\begin{equation}\label{eq-omegaep}
\omega_\epsilon:=\mu^* \omega+\epsilon \omega_P
\end{equation}
is a K\"{a}hler form and one can define Darvas' Finsler metric $d_{1,\epsilon}$ on $\mcH(Z, \omega_\epsilon)$. Note that $u\in \mcH(Z, \omega)$ implies $u\in \mcH(Y, \omega_\epsilon)$. One then defines (see \cite[Definition 1.10]{DNG18})
\begin{eqnarray*}
d_1(u_0, u_1)=\liminf_{\epsilon\rightarrow 0}d_{1,\epsilon}(u_0, u_1).
\end{eqnarray*}

It is known that $u_j\rightarrow u$ in $\mcE^1$ under the strong topology 
if and only if $d_1(u_j, u)=0$.  Moreover in this case the Monge-Amp\`{e}re measures $(\sddb(\psi+u_j))^n$ converges weakly to $(\sddb(\psi+u))^n$
(see \cite[Proposition 2.6]{BBEGZ}).

\subsection{Energy functions}\label{sec-energy}


For any $\vphi\in {\rm PSH}([\omega])$ such that $\vphi-\psi\in \mcE^1(L)$, we also have the following well-studied functionals: 
\begin{eqnarray}
\bfJ(\vphi)&:=&\bfJ_{\psi}(\vphi)=\int_Z (\vphi-\psi)(\sddb\psi)^n-\bfE_{\psi}(\vphi), \label{eq-Jphi}\\
\bfI(\vphi)&:=&\bfI_{\psi}(\vphi)= \int_Z (\vphi-\psi)\left((\sddb\psi)^n-(\sddb\vphi)^n\right), \label{eq-Iphi}\\
({\bf I}-\bfJ)(\vphi)&:=&({\bf I}-\bfJ)_{\psi}(\vphi)=\bfE_{\psi}(\vphi)-\int_Z (\vphi-\psi)(\sddb \vphi)^n.
\end{eqnarray}
We have the well-known inequality:
\begin{equation}\label{eq-IJineq}
\frac{1}{n+1}\bfI\le \bfJ\le \frac{n}{n+1}\bfI.
\end{equation}

Let $\mu: Y\rightarrow Z$ be a log resolution of singularities such that $\mu^{-1}Z^\sing=\sum_k E_k$ is the reduced exceptional divisor, $Q':=\mu^{-1}_*Q$ is the strict transform of $Q$ and $Q'+\sum_k E_k$ has simple normal crossings. We can write:
\begin{equation}
K_Y+Q'=\mu^*(K_Z+Q)+\sum_k a_k E_k.
\end{equation}
\begin{defn}
$(Z, Q)$ is said to have sub-klt singularities if there exists a log resolution of singularities as above such that $a_k>-1$ for all $k$. If $Q$ is moreover effective, then $(Z, Q)$ is said to have klt singularities.
\end{defn}
Fix $\ell_0\in \bN^*$ such that $\ell_0 (K_Z+Q)$ is Cartier. If $\sigma$ is a nowhere-vanishing holomorphic section of the corresponding line bundle over a smooth open set $U$ of $Z$, then there is a pull-back meromorphic volume form on $\mu^{-1}(U)$:
\begin{equation}
\mu^*\left(\sqrt{-1}^{\ell_0 n^2} \sigma\wedge \bar{\sigma}\right)^{1/\ell_0}=\prod_{i}|z_i|^{2a_i}dV,
\end{equation}
where $\{z_i\}$ are local holomorphic coordinates and $dV$ is a smooth volume for on $Y$. If $(Z, Q)$ is sub-klt, then the above volume form is locally integrable. 
\begin{defn}[{see \cite[section 3]{BBEGZ}}]
Let $(Z, Q)$ be a sub-klt pair. Assume that $L=\lambda^{-1}(-K_Z-Q)$ is an ample $\bQ$-line bundle for $\lambda>0\in \bQ$.
Let $\vphi\in \PSH_\bd(Z, L)$ be a bounded Hermitian metric on the $\bQ$-line bundle $L$. The adapted measure $\mes_\vphi$ is a globally defined measure:
\begin{equation}\label{eq-mesvphi}
\frac{e^{-\lambda\vphi}}{|s_Q|^2}:=
\mes_\vphi=\left(\sqrt{-1}^{\ell_0 n^2} \sigma\wedge \bar{\sigma}\right)^{1/\ell_0}{|\sigma^*|_{\ell_0 \lambda \vphi}^{2/\ell_0}},
\end{equation}
where $\sigma^*$ is the dual nowhere-vanishing section of $-\ell_0(K_Z+Q)$. 
\end{defn}
Note that over the locus where $Q$ is $\bQ$-Cartier, $s_Q$ in the above identity has the natural meaning of being the defining section of the divisor $Q$. In particular, the right-hand-side of \eqref{eq-mesvphi} defines a measure on the regular locus (i.e. simple normal crossing locus) of $(Z, Q)$, which naturally extends to be a measure on the whole $Z$.

The Ding- and Mabuchi- functionals on $\mcE^1(Z, L)$ are defined as follows:
\begin{eqnarray}
\def\arraystretch{1.5}
\bL(\vphi)&:=&\bL_{(Z,Q)}(\vphi)=-\frac{V}{\lambda}\cdot \log\left(\int_Z e^{- \lambda\vphi}\frac{1}{|s_Q|^2}\right)\label{eq-LB}\\
\bfD(\vphi)&:=&\bfD_{(Z,Q),\psi}(\vphi)=\bfD_{\psi}(\vphi)=-\bfE_{\psi}(\vphi)+\bL_{(Z,Q)}(\vphi) \label{eq-DB}\\
\bfH(\vphi)&:=&\bfH_{(Z,Q),\psi}(\vphi)=\int_Z \log\frac{|s_Q|^2(\sddb\vphi)^n}{e^{-\lambda \psi}}(\sddb \vphi)^n \label{eq-Hphi}\\
\bfM(\vphi)&:=&\bfM_{(Z,Q),\psi}(\vphi)=\bfM_{\psi}(\vphi)=\lambda^{-1} \bfH(\vphi)-({\bf I}-\bfJ)_{\psi}(\vphi) \label{eq-Mphi}.
\end{eqnarray}

\begin{defn}[{\cite[Definition 1.3]{BBEGZ}}]
A positive measure $\nu$ on $Z$ is tame if $\nu$ puts no mass on closed analytic sets and if there is a resolution of singularities $\mu: Y\rightarrow Z$ such that the lift $\nu_Y$ of $\nu$ to $Y$ has $L^p$ density for some $p>1$.
\end{defn}
The following compactness result is very important in the variational approach for solving K\"{a}hler-Einstein equations using pluripotential theory.
\begin{thm}[{\cite[Theorem 2.17]{BBEGZ}}]\label{thm-BBEGZ}
Let $\nu$ be a tame probability measure on $Z$. 
For any $C>0$, the following set is compact in the strong topology:
\[
\left\{u\in \mcE^1(Z, \omega); \quad
\sup_{Z}u=0, \quad \int_Z \log\frac{\omega_{u}^n}{\nu}\omega^n_{u}<C
\right\}.
\]
\end{thm}
In the rest of this subsection, we will assume that $(Z, Q)=(X, D)$ is a log Fano pair in the following sense:
\begin{defn}\label{defn-logFano}
A pair $(X, D)$ is called a log Fano pair if $-(K_X+D)$ is an ample $\bQ$-Cartier divisor, $D$ is effective and $(X, D)$ has klt singularities.
\end{defn}
We have the following well-known definition:
\begin{defn}
We say that the energy $\bfF\in \{\bfD, \bfM\}$ is proper (sometimes called coercive in the literature) if there exist $\gamma>0$ and $C\in \bR$ such that for any $\vphi\in \mcE^1(X, L)$
\begin{equation}\label{eq-FGproper}
\bfF(\vphi) \ge \gamma\cdot  \bfJ(\vphi)-C.
\end{equation}
\end{defn}
We will use the following analytic criterion for the existence of K\"{a}hler-Einstein metrics on log Fano varieties.
\begin{thm}[{\cite{BBEGZ}, \cite{Dar17}, \cite{DNG18}}]\label{thm-analytic}
Let $(X, D)$ be a log Fano pair with a discrete automorphism group. Assume $L=\lambda^{-1}(-K_X-D)$ with $\lambda>0\in \bQ$. The following conditions are equivalent:
\begin{enumerate}[(1)]
\item The Ding energy $\bfD$ is proper over $\mcE^1(X, L)$. 
\item The Ding energy $\bfD$ is proper over $\mcH(X, L)$.
\item The Mabuchi energy $\bfM$ is proper over $\mcE^1(X, L)$.
\item $(X, D)$ admits a unique K\"{a}hler-Einstein metric with Ricci curvature $\lambda$.
\end{enumerate}

\end{thm}

For later use, we need a refinement of the properness for $\bfM$. Fix a log resolution $\mu: Y\rightarrow X$ and assume that the reduced exceptional divisor is given by $\mu^{-1}X^{\rm sing}=\sum_k E_k$ and write
\begin{equation}
-K_Y=\mu^*(-(K_X+D))+\mu_*^{-1}D+\sum_k (-a_k) E_k=:\mu^*(K_X+D)+B.
\end{equation} 
Choose a smooth reference metric $\psi_0\in \mcH(X, \omega)$. 
By abuse of notations, we will identify Hermitian metrics on $L:=-(K_X+D)$ with their pull-back metric on $\mu^*L$. As a consequence, we identify
 $\mcE^1(X, \omega)$ with $\mcE^1(Y, \mu^*\omega)$. 
So for any $\vphi\in \mcE^1(X, \omega)$, we have the following identities:
\begin{eqnarray}
\bfM(\vphi)&=&\int_X \log\frac{(\sddb \vphi)^n}{e^{-\psi_0}\frac{1}{|s_D|^2}}(\sddb \vphi)^n+\int_X(\vphi-\psi_0)(\sddb\vphi)^n-\bfE_{\psi_0}(\vphi)\nonumber \\
&=&\int_Y \log\frac{(\sddb \vphi)^n}{e^{-\psi_0}\frac{1}{|s_B|^2}}(\sddb \vphi)^n+\int_Y(\vphi-\psi_0)(\sddb\vphi)^n-\bfE_{\psi_0}(\vphi)\nonumber \\
&=&\int_Y \log \frac{(\sddb\vphi)^n}{\Omega}(\sddb\vphi)^n-\int_Y\log \frac{e^{-\psi_0}}{|s_B|^2 \Omega}(\sddb\vphi)^n\nonumber \\
&&\hskip 4.5cm +\int_Y (\vphi-\psi_0)(\sddb\vphi)^n-\bfE_{\psi_0}(\vphi), \label{eq-Mabuchi1}
\end{eqnarray}
where for the last identity we used a fixed smooth volume form $\Omega$ on $Y$.
Let $s_B$ be the defining section of the $\bQ$-line bundle associated to the divisor $B$ and choose a smooth Hermitian metric on this line bundle.
Consider the space:
\begin{eqnarray}
\hat{\cH}(\omega)&:=&\hat{\cH}(X, \omega)=\left.\{u\in \PSH_\bd(\omega); (\mu^*u)|_{Y\setminus B}\in C^\infty(Y\setminus B), \right. \frac{\mu^*(\sddb(\psi+u))^n}{\Omega} \in C^\infty(Y),\nonumber  \\
&&\left. \text{ and there exist } \alpha>0, C>0 \text{ such that } |\sddb (\mu^*u)|_{\omega}\le C |s_B|^{-\alpha} \text{ on } Y\setminus B \right\}. \nonumber \\
\hat{\cH}(L)&:=&\hat{\cH}(X, L)=\{\psi+u; u\in \hat{\cH}(X, \omega)\}. \label{eq-hatcH}
\end{eqnarray}
\begin{prop}\label{prop-proper2}
With the same notations as in the above theorem, the following conditions are equivalent:
\begin{enumerate}[(1)]
\item The Mabuchi energy $\bfM$ is proper over $\mcE^1(X, L)$.
\item The Mabuchi energy $\bfM$ is proper over $\hat{\cH}(X, L)$.
\item $(X, D)$ admits a unique K\"{a}hler-Einstein metrics with Ricci curvature $\lambda$.
\end{enumerate}
\end{prop}
\begin{proof}
By the above theorem, we just need to show that (2) implies (1). For any $\vphi\in \mcE^1(L)$, we only need to show that there exists $\vphi_j\in \hat{\cH}(X, \omega)$ such that $\vphi_j$ $d_1$-converges to $\vphi$ and $\bfM(\vphi_j)\rightarrow \bfM(\vphi)$. To prove this, we carry verbatim the argument Berman-Darvas-Lu in \cite[Proof of Lemma 3.1]{BDL17}, which we just sketch here. We can assume that the entropy of $(\sddb\vphi)^n$ is finite, since otherwise the inequality \eqref{eq-FGproper} with $\bfF=\bfM$ is trivially true. Set $g=\frac{(\sddb\vphi)^n}{\Omega}$ and $h_k=\min\{k, g\}$. Then $\|h_k-g\|_{L^1(\Omega)}\rightarrow 0$ and by the dominated convergence theorem
\begin{equation}
\int_Y h_k (\log h_k)\Omega \longrightarrow \int_Y g(\log g)\Omega.
\end{equation}
By using the density of $C^\infty(Y)$ in $L^1(\Omega)$ and dominated convergence theorem, there is a sequence of positive functions $g_k\in C^\infty(Y)$ such that $\|g_k-h_k\|_{L^1}\le \frac{1}{k}$ and
\begin{equation}
\left|\int_Y h_k (\log h_k)\Omega-\int_Y g_k (\log g_k)\Omega\right|\le \frac{1}{k}.
\end{equation}
As a consequence we get 
$\|g-g_k\|_{L^1}\rightarrow 0$ and
\begin{equation}
\int_Y g_k (\log g_k)\Omega\longrightarrow \int_Y g (\log g)\Omega=\int_Y \log\frac{(\sddb \vphi)^n}{\Omega}(\sddb \vphi)^n.
\end{equation}
Using the Calabi-Yau type results for degenerate complex Monge-Amp\`{e}re equations as proved in \cite[Theorem 3.5]{EGZ09} or \cite[Theorem 6.1]{DP10} (see Proposition \ref{prop-partialest}), we find potentials $v_k\in \hat{\cH}(\omega)$ with $\sup_{Y} v_k=0$ and $\omega^n_{v_k}=C_k\cdot g_k\Omega$. 

By construction the entropy of $(\sddb(\psi_0+v_k))^n$ converges to the entropy of $(\sddb\vphi)^n$. On the other hand, Theorem \ref{thm-BBEGZ} implies that $v_k$ converges strongly to some $v\in \mcE^1(Y, \omega)=\mcE^1(X, \omega)$. In particular, the Monge-Amp\`{e}re measures $\omega^n_{v_k}$ converge weakly to $\omega^n_v$. So we get the identity $\omega_v^n=g\Omega$. By the uniqueness of the solution to the complex Monge-Amp\`{e}re equations, we get $v=\vphi-\psi_0$ up to a constant. Then the convergence of the other part of Mabuchi energy follows as in the proof from \cite[4.2]{BDL17} (see also \cite{Dar17}).

\end{proof}

\begin{rem}\label{rem-smoothproper}
Note that by the work of Tian \cite{Tia97} and Darvas-Rubinstein \cite{DR17}, for smooth Fano manifolds with discrete automorphism groups, the properness of Mabuchi energy over $\mcH(X, \omega)$ is equivalent to the existence of KE metrics.
However currently it is not known yet whether this is true over a singular Fano variety $X$. 
The difficulty is that it is not clear whether the properness of Mabuchi energy over $\mcH(X, \omega)$ can imply its properness over $\mcE^1(X, \omega)$. In this respect there is an imprecision in the statement of \cite[Theorem 2.2]{Dar17}, and we would like to thank T. Darvas for clarifications regarding this point.
\end{rem}

\subsection{Stability via test configurations}\label{sec-TC}

In this section we recall the definition of test configurations and stability of log Fano varieties.
\begin{defn}[{\cite{Tia97, Don02}, see also \cite{LX14}}]\label{defn-TC}
Let $(Z, Q, L)$ be as before.
\begin{enumerate}[(1)]
\item A test configuration of $(Z, L)$, denoted by $(\mcZ, \mcL, \eta)$ or simply by $(\mcZ, \mcL)$, consists of the following data
\begin{itemize}
\item A variety $\mcZ$ admitting a $\bC^*$-action, which is generated by a holomorphic vector field $\eta$, and a $\bC^*$-equivariant morphism $\pi: \mcZ\rightarrow \bC$, where $\bC^*$ acts on $\bC$ by the standard multiplication.
\item A $\bC^*$-equivariant $\pi$-semiample $\bQ$-Cartier divisor $\mcL$ on $\mcY$ such that there is a $\bC^*$-equivariant isomorphism $\mathfrak{i}_\eta: (\mcZ, \mcL)|_{\pi^{-1}(\bC\backslash\{0\})}\cong (Z, L)\times \bC^*$.
\end{itemize}
Let $\mcQ:=Q_{\mcZ}$ denote the closure of $Q\times\bC^*$ in $\mcZ$ under the inclusion $Q\times\bC^*\subset Z\times\bC^* \stackrel{\mathfrak{i}_\eta}{\cong} \mcZ\times_\bC\bC^*\subset \mcZ$. We say that $(\mcZ, \mcQ, \mcL)$ is a test configuration of $(Z, Q, L)$. 

Denote by $\bar{\pi}: (\bar{\mcZ}, \bar{\mcQ}, \bar{\mcL})\rightarrow \bP^1$ the natural equivariant compactification of $(\mcZ, \mcQ, \mcL)\rightarrow \bC$ obtained by using the isomorphism $\mathfrak{i}_\eta$ and then adding a trivial fiber over $\{\infty\}\in \bP^1$. 

\item
A test configuration is called normal if $\mcZ$ is a normal variety. We will always consider normal test configurations in this paper. 

A test configuration $(\mcZ, \mcQ, \mcL)$ is called dominating if there exists a $\bC^*$-equivariant birational morphism $\rho: (\mcZ, \mcQ) \rightarrow (Z, Q)\times\bC$.

Two test configurations $(\mcZ_i, \mcQ_i, \mcL_i), i=1, 2$ are called equivalent, if there exists a family $(\mcZ_3, \mcQ_3)$ that $\bC^*$-equivariantly dominates both test configurations via $q_i: (\mcZ_3, \mcQ_3)\rightarrow (\mcZ_i, \mcQ_i)$, $i=1,2$ and satisfies $q_1^*\mcL_1=q_2^*\mcL_2$. Note that any test configuration is equivalent to a dominating test configuration.
\item
Assume $L=\lambda^{-1}(-K_Z-Q)$. 
For any normal test configuration $(\mcZ, \mcQ, \mcL)$ of $(Z, Q, L)$, define the divisor 
$\Delta_{(\mcZ, \mcQ, \mcL)}$ to be the $\bQ$-divisor supported on $\mcZ_0$ that is given by:
\begin{equation}
\Delta:=\Delta_{(\mcZ, \mcQ, \mcL)}=-K_{\mcZ/\bC}-\mcQ-\lambda \cdot \mcL.
\end{equation}

\item
Assume that $(Z, Q)$ is a log Fano pair and $L=\lambda^{-1}(-K_Z-Q)$ for some $\lambda>0\in \bQ$. A test configuration of $(Z, Q, L)$ is called a special test configuration, if the following conditions are satisfied:
\begin{itemize}
\item $\mcZ$ is normal, and $\mcZ_0$ is an irreducible normal variety;
\item $\mcL\sim_{\bC} \lambda^{-1} (-K_{\mcZ/\bC}-\mcQ)$, which is a $\pi$-ample $\bQ$-Cartier divisor;
\item $(\mcZ, \mcZ_0+\mcQ)$ has plt singularities.
\end{itemize}

\end{enumerate}

\end{defn}

For any (dominating) normal test configuration $(\mcZ, \mcQ, \mcL)$ of $\left(Z, Q, L=\lambda^{-1}(-K_Z-Q)\right)$, we attach the following well-known invariants (where $V=(2\pi)^n L^{\cdot n}$): 
\begin{eqnarray}
\bfE^\NA(\mcZ,  \mcL)&=&
\frac{\left(\bar{\mcL}^{\cdot n+1}\right)}{n+1},\label{eq-ENA} \\
\bfJ^\NA(\mcZ, \mcL)&=&
\left(\bar{\mcL}\cdot \rho^*(L\times\bP^1)^{\cdot n}\right)- \frac{\left(\bar{\mcL}^{\cdot n+1}\right)}{n+1}, \label{eq-JNA} \\
\CM(\mcZ, \mcQ, \mcL)&=&\frac{1}{\lambda} K_{(\bar{\mcZ}, \bar{\mcQ})/\bP^1}\cdot \bar{\mcL}^{\cdot n}+\frac{n}{n+1}\bar{\mcL}^{\cdot n+1}, \label{eq-defCM} \\
\bL^\NA(\mcZ, \mcQ, \mcL)&=&\frac{V}{\lambda}\cdot\left( \lct\left(\mcZ, \mcQ+\Delta; \mcZ_0\right)-1\right),\\
\bfD^\NA(\mcZ, \mcQ, \mcL)&=&\frac{- \bar{\mcL}^{\cdot n+1}}{n+1}+\bL^\NA(\mcZ, \mcQ, \mcL). \label{eq-bfDTC}
\end{eqnarray}

The following result is now well known:
\begin{prop}[{see \cite{Berm15, BHJ19}}]\label{prop-BHJslope}
Let $(\mcZ, \mcQ, \mcL)$ be a normal test configuration of $(Z, Q, L)$. Let $\Phi=\{\vphi(s)\}$ be a locally bounded and positively curved Hermitian metric on $\mcL$. Then the following limits hold true:
\begin{equation}
\lim_{s\rightarrow+\infty} \frac{\bfF(\vphi(s))}{s}=\bfF^\NA(\mcZ, \mcQ, \mcL),
\end{equation}
where the energy $\bfF$ is any one from $\{\bfE, \bfJ, \bfL, \bfD\}$.
\end{prop}

\begin{defn}\label{defn-uniformK}
\begin{enumerate}[(1)]
\item
$(Z, Q)$ is called {\it uniformly K-stable} if there exists $\gamma>0$ such that $\CM(\mcZ, \mcQ, \mcL)\ge \gamma\cdot \bfJ^\NA(\mcZ, \mcL)$ for any normal test configuration $(\mcZ, \mcQ, \mcL)$ of $(Z, Q, L)$.

\item
$(Z, Q)$ is called {\it uniformly Ding-stable} if there exists $\gamma>0$ such that $\bfD^\NA(\mcZ, \mcQ, \mcL)\ge \gamma\cdot \bfJ^\NA(\mcZ,  \mcL)$ for any normal test configuration $(\mcZ, \mcQ, \mcL)$ of $(Z, Q, L)$. 

For convenience, we will call $\gamma$ to be a slope constant.

\end{enumerate}

\begin{rem}\label{rem-lambda}
The rescaling parameter $\lambda$ is included in our discussion to make the following arguments more flexible. By checking the rescaling properties of functionals, it is easy to see that if the above statements hold for one $\lambda>0$ then they hold for any $\lambda>0$ with the same slope constant.
\end{rem}

\end{defn}
For any special test configuration $(\mcZ^s, \mcQ^s, \mcL^s)$, its CM weight coincides with its $\bfD^\NA$ invariant, which coincides with the original Futaki invariant of the central fibre (as generalized by Ding-Tian):
\begin{equation}\label{eq-CMstc}
{\rm CM}(\mcZ^s, \mcQ^s, \mcL^s)=
\bfD^\NA(\mcZ^s, \mcQ^s, \mcL^s)=
-\frac{(-K_{(\overline{\mcZ^s}, \overline{\mcQ^s})/\bP^1})^{\cdot n+1}}{n+1}=\Fut_{(\mcZ^s_0, \mcQ^s_0)}(\eta).
\end{equation}

By the work in \cite{BBJ15, Fuj16} (see also \cite{LX14}), to test uniform K-stability, one only needs to test on special test configurations. As a consequence,
 \begin{thm}[\cite{BBJ15, Fuj16}]\label{thm-BBJ}
For a log Fano pair $(X, D)$, $(X, D)$ is uniformly K-stable if and only if $(X, D)$ is uniformly Ding-stable.
\end{thm}

\subsection{Non-Archimedean functionals and valuative criterion}\label{sec-valuative}

Here we briefly recall the non-Archimedean formulation of K-stability/Ding-stability following the works in \cite{BBJ18, BHJ17, BoJ18a}. Let $(Z, Q, L)$ be the polarized projective pair as before. We denote by $(Z^\NA, Q^\NA, L^\NA)$ the Berkovich analytification of $(Z, Q, L)$ with respect to the trivial absolute value on the ground field $\bC$. $Z^\NA$ is a topological space, whose points can be considered as semivaluations on $Z$, i.e. valuations $v: \bC(W)^*\rightarrow \bR$ on function field of subvarieties $W$ of $Z$, trivial on $\bC$. The topology of $Z^\NA$ is generated by functions
of the form $v\mapsto v(f)$ with $f$ a regular function on some Zariski open set $U\subset Z$. One can show that $Z^\NA$ is compact and Hausdorff.
Let $Z^{\rm div}_\bQ$ be the set of valuations over $Z$ which are of the form $\lambda\cdot \ord_E$ where $\ord_E\in Z^{\rm div}$ is a divisorial valuation over $Z$.
Then $Z^{\rm div}_\bQ \subset Z^\NA$ is dense.

For any $v\in Z^{\rm div}_\bQ$, let $G(v)$ denote the standard Gauss extension:
for any $f=\sum_{i\in \bZ} f_i t^i\in \bC(Z\times\bC)$ with $f_i\in \bC(Z)$, set
\begin{equation}
G(v)\left(\sum_i f_i t^i\right)=\min_i \{v(f_i)+i\}.
\end{equation}
In this paper, we will identify a non-Archimedean metric with its associated function on $Z^{\rm div}_\bQ$.
\begin{defn}\label{defn-TCNA}
Let $(\mcZ, \mcL)$ be a dominating test configuration of $(Z, L)$ with $\rho: \mcZ\rightarrow Z\times\bC$ being a $\bC^*$-equivariant morphism.  The non-Archimedean metric defined by $(\mcZ, \mcL)$ is represented by the following function on $Z^{\rm div}_\bQ$:
\begin{equation}
\phi_{(\mcZ, \mcL)}(v)=G(v)\left(\mcL-\rho^*(L\times\bC)\right).
\end{equation}
The set of non-Archimedean metrics obtained in this way will be denoted as $\mcH^\NA(Z, L)$.
If $(\mcZ, \mcL)$ is obtained as the normalized blowups of $(Z, L)\times \bC$ along some flag ideal sheaf $\mcI$:
\begin{equation}
\mcZ=\text{ normalization of } Bl_\mcI (Z\times\bC), \quad \mcL=\pi^*(L\times\bC)-c E
\end{equation}
for some $c\in \bQ>0$, where $\pi: \mcZ\rightarrow Z\times\bC$ is the natural projection and $E$ is the exceptional divisor of blowup, then we have:
\begin{equation}\label{eq-phimcI}
\phi_{\mcI}(v):=-G(v)(c E)=-c\cdot G(v)(\mcI), \quad \sup \phi_{\mcI}=0.
\end{equation}
\end{defn}
\begin{defn}[{\cite[section 4]{BBJ18}}]\label{defn-PhiNA}
A psh ray $\Phi=\{\vphi(s)\}$ has linear growth if
\begin{equation}
\lim_{s\rightarrow+\infty} s^{-1} \sup_X (\vphi(s)-\psi_0)<+\infty.
\end{equation}
In this case, we associate a function:
$
\Phi^\NA: Z^{\rm div}_\bQ\rightarrow \bR, 
$
by setting for any $v\in Z^{\rm div}_\bQ$, 
$$\Phi^\NA(v)=-G(v)(\Phi),$$
which is the negative of the generic Lelong number of $\Phi$ on a suitable blow-up where the center of $G(v)$ is of codimension 1.
\end{defn}
\begin{exmp}\label{exmp-TCsubgeo}
For each normal test configuration $(\mcZ, \mcL)$ of $(Z, L)$, one can find a psh ray with linear growth $\Phi_{(\mcZ, \mcL)}$ such that $\Phi^\NA_{(\mcZ, \mcL)}=\phi_{(\mcZ, \mcL)}$. Indeed, it is well-known that for $m\gg 1$ sufficiently divisible, one has an equivariantly embedding $\iota_m: \mcZ\rightarrow \bP^{N_m}\times\bC$ with $N_m+1=\dim_{\bC} H^0(Z, mL)$ and $\iota_m^*\cO_{\bP^{N_m}}(1)\sim_\bC m \mcL$ where $\bC^*$ acts on $(\bP^{N_m}, \mcO_{\bP^{N_m}}(1))$ via a one-parameter subgroup in $GL(N_m+1,\bC)$. Then the psh ray can be chosen to be $\iota_m^*\left(p_1^*\frac{1}{m}\psi_{\rm FS}\right)$ where $\psi_{\rm FS}$ is the canonical Fubini-Study metric on $\cO_{\bP^{N_m}}(1)$ and $p_1$ is the projection to the first factor.
\end{exmp}
For any $\phi\in \mcH^\NA(Z, L)$, the non-Archimedean functionals can be defined formally as (where $V=(2\pi)^n L^{\cdot n}$):
\begin{eqnarray}
\bfE^\NA(\phi)&:=&\bfE^\NA_{L}(\phi)=\frac{1}{n+1}\sum_{j=0}^n \int_{X^\NA}\phi(\omega^\NA_\phi)^j\wedge (\omega^\NA)^{n-j},\\
\bfJ^\NA(\phi)&:=&\bfJ^\NA_L(\phi)=V\cdot \sup \phi-\bfE^\NA(\phi), \label{eq-JNAphi}\\
&&\nonumber \\
\bL^\NA(\phi)&:=&\bL^\NA_{(Z,Q)}(\phi)=\frac{V}{\lambda} \cdot \inf_{v\in Z^{\rm div}_\bQ} \left(A_{(Z, Q)}(v)+\lambda \phi(v)\right)\\
\bfD^\NA(\phi)&:=&\bfD^\NA_{(Z, Q)}(\phi)=-\bfE^\NA(\phi)+\bfL^\NA(\phi).
\end{eqnarray}
They recover the non-Archimedean functional for test configurations: for functional ${\bf F}$ appearing in \eqref{eq-ENA}-\eqref{eq-bfDTC}:
$
{\bf F}^\NA(\phi_{(\mcZ, \mcL)})={\bf F}^\NA(\mcZ, \mcQ, \mcL).
$

We will need the valuative criterion for the uniform Ding-stability studied in \cite{BoJ18b, Fuj16}. Let $Z$ be a projective variety with polarization $L$. 
For any divisorial valuation $\ord_E$ over $Z$, let $\mu': Y'\rightarrow Z$ be a birational morphism such that $E$ is an irreducible Weil divisor on $Y'$. 
Set \begin{eqnarray*}
\vol(L-x E)&:=&\lim_{m\rightarrow+\infty} \frac{h^0(Y', m\mu'^*L-\lceil mx\rceil E)}{m^n/n!} \quad \text{ for any } x\in \bR,\\
S_{L}(E)&:=&\frac{1}{L^n}\int_0^{+\infty}\vol(L-xE)dx.
\end{eqnarray*}
Following \cite{Fuj16, BlJ17, BoJ18b}, we define
the stability threshold as:
\begin{equation}\label{eq-deltaXY}
\delta(Z, Q):=\inf_{\ord_E\in Z^{\rm div}} \frac{A_{(Z, Q)}(E)}{S_{-K_Z-Q}(E)}.
\end{equation}
In light of the work \cite{BoJ18b}, the following criterion could be derived by using the non-Archimedan version of the equivalence between properness of Mabuchi energy and Ding energy. Note that such type of criterion for K-(semi)stability by using valuations first appeared in the first author's work \cite{Li15} and also in \cite{Fuj16}.
\begin{thm}\label{thm-MvsD}
\begin{enumerate}
\item {\rm (\cite{Fuj16, Fuj17a})} Let $(X, D)$ be a log Fano pair. Then $(X, D)$ is uniformly Ding-stable if and only if $\delta(X, D)>1$. 
\item {\rm (see \cite[Theorem 7.3]{BBJ18})} 
Assume $(Y, Q)$ is a sub-klt pair with $Y$ smooth. Then
$(Y, Q)$ is uniformly Ding-stable if and only if $\delta(Y, Q)>1$.
Moreover, if $\delta(Y, Q)> 1$, then $\bfD_{(Y, Q)}^\NA\ge \left(1-\delta(Y, Q)^{-1/n}\right)\bfJ^\NA$ on $\mcH^\NA(L)$.

\end{enumerate}
\end{thm}
Fujita proved the first item using purely algebro-geometric techniques (e.g. MMP as used in \cite{LX14}) which also work well for any $\bQ$-Fano variety. The proof of the second item is based on Boucksom-Jonsson's non-Archimedean formulation (\cite{BoJ18a, BoJ18b}). 
We emphasize that the key feature of the second item is that $Q$ is allowed to be non-effective, and this will be very important for our argument. 
On the other hand, since in \cite{BBJ18} the twisting is assumed to be a klt current which is by definition quasi-positive, we give its proof  following \cite[Proof of Theorem 7.3]{BBJ18}
to show that it indeed works for the non-effective $Q$ at hand.
\begin{proof}[Proof of \ref{thm-MvsD}.2]
Because of the remark \ref{rem-lambda}, we can assume $\lambda=1$ so that $L=-K_Y-Q$.
Assume $\delta:=\delta(Y, Q)>1$. By assumption $A_{(Y,Q)}(v)\ge \delta S_L(v)$ for any divisorial valuation $v\in Y^{\rm div}_{\bQ}$. Pick any $\phi\in \cH^\NA(Y, L)$ with  $\sup\phi=0$. Since $\delta\ge 1$, $\delta^{-1}\phi\in \cH^\NA(Y, L)$. Then by \cite[Proposition 7.5]{BoJ18a}, we have the identity:
\begin{equation}
\bfE^\NA_{L}(\delta^{-1}\phi)\le V\cdot \inf_{v\in Y^{\rm div}_\bQ}\left(S_{L}(v)+\delta^{-1}\phi(v)\right).
\end{equation}
Note that here we are working on a smooth $Y$ as in \cite{BBJ18}. 
So we get:
\begin{eqnarray*}
\bL_{(Y, Q)}^\NA(\phi)&=&V\cdot \inf_{v\in Y^{\rm div}_\bQ}\left(A_{(Y,Q)}(v)+\phi(v)\right)\ge V\cdot \inf_{v\in Y^{\rm div}_\bQ}\left(\delta S_{L}(v)+\phi(v)\right) \\
&\ge & \delta \bfE^\NA(\delta^{-1}\phi).
\end{eqnarray*}
So we get:
\begin{eqnarray*}
\bfD^\NA_{(Y,Q)}(\phi)&\ge& \delta \bfE^\NA(\delta^{-1}\phi)-\bfE^\NA(\phi)=\delta \bfJ^\NA(\delta^{-1}\phi)-\bfJ^\NA(\phi)\\
&\ge& (1-\delta^{-1/n})\bfJ^\NA(\phi).
\end{eqnarray*} 
The last inequality is the non-Archimedean version of Ding's inequality and is proved in \cite[Lemma 6.17]{BoJ18a}. 

Conversely, assume $(Y, Q)$ is uniformly Ding-stable. By using the definition \ref{defn-uniformK}, there exists $\gamma>0$ such that for any $\phi\in \mcH^\NA(Y, L)$: 
\begin{eqnarray*}
\bfL^\NA(\phi)=\bfD^\NA(\phi)+\bfE^\NA(\phi)\ge \gamma\cdot \bfJ^\NA(\phi)+\bfE^\NA(\phi).
\end{eqnarray*}
As shown in \cite{BBJ18, BoJ18b}, the above inequality actually holds for any finite energy non-Archimdedean metrics. In particular, it holds for $\phi=\phi_v$ which satisfies $\phi_v(v)=0$ and
\begin{equation}
\frac{1}{V}\bfE^\NA(\phi_v)=S_L(v), \quad \frac{1}{n}S_L(v)\le \frac{1}{V}\bfJ^\NA(\phi_v)\le n S_L(v).
\end{equation}
On the other hand, $A_{(Y,Q)}(v)\ge \inf_{v\in Y^{\rm div}_\bQ}(A_{(Y,Q)}(v)+\phi_v(v))=\frac{1}{V}\bL^\NA(\phi_v)$. So we easily get:
\begin{equation}
A_{(Y,Q)}(v)\ge (1+\gamma n^{-1})S_L(v).
\end{equation}

\end{proof}

\section{Proof of the main result}\label{sec-proof}

The rest of this paper is to prove Theorem \ref{thm-main} following the argument sketched in section \ref{sec-skmod}.

\subsection{Step 1: Constructing a destabilizing geodesic ray}\label{sec-step1}

\subsubsection{Basic construction}

The argument in this section is the same as the construction \cite{BBJ15, BBJ18} which is inspired by the earlier work in \cite{DH17, DR17} (see also \cite{Don99}).  All energy functionals in this step are on $X$ itself as defined in \eqref{eq-Ephi}-\eqref{eq-Mphi}.
Recall that by Proposition \ref{prop-proper2}, we just need to prove that the Mabuchi energy  $\bfM=\bfM_{\psi_0}$ (see \eqref{eq-Mphi}) is proper over $\hat{\mcH}(X, L)$. 

 Assume on the contrary that $\bfM=\bfM_{\psi_0}$ is not proper with slope constant $\gamma$. Then we can pick a sequence $\{u_j\}_{j=1}^\infty\in \hat{\mcH}(X, \omega)$ such that $\vphi_j={\psi_0}+u_j$ satisfies:
\[
\bfM(\vphi_j)\le \gamma \bfJ(\vphi_j)-j.
\]
We will choose $\gamma$ to be small in the last step of proof in section \ref{sec-step5} to get a contradiction.

We normalize $\vphi_j$ such that $\sup (\vphi_j-{\psi_0})=0$. The inequality $\bfM\ge C-n \bfJ$ implies $\bfJ(\vphi_j)\rightarrow+\infty$, and hence $\bfE(\vphi_j)\le -\bfJ(\vphi_j)\rightarrow-\infty$.

Denote $V=(2\pi)^n (-K_X-D)^{\cdot n}$. By Proposition \ref{prop-geod} (see \cite{Dar17, DNG18}), we can connect ${\psi_0}$ and $\vphi_j$ by a geodesic segment $\{\vphi_j(s); 0\le s\le S_j\}$ parametrized so that $S_j=-\bfE(\vphi_j)\rightarrow +\infty$. For any $s\in (0, S_j]$, we have $\bfE(\vphi_j(s))=-s$ and $\sup(\vphi_j(s)-{\psi_0})=0$.
So $\bfJ(\vphi_j(s))\le V\cdot \sup(\vphi_j(s)-{\psi_0})-\bfE(\vphi_j(s))=s \le S_j$ and $\bfM(\vphi_j)\le \gamma \cdot S_j-j\le \gamma  S_j$.
By Proposition \ref{prop-Mabconv}, we get the inequality:
\begin{equation}
\bfM(\vphi_j(s))\le \frac{S_j-s}{S_j}\bfM({\psi_0})+\frac{s}{S_j}\bfM(\vphi_j)\le \gamma s+C.
\end{equation}
Using $\bfM\ge \bfH-n \bfJ$, we get $\bfH(\vphi_j(s))\le (\gamma+n)s+C$. So for any fixed $S>0$ and $s\le S$, the metrics $\vphi_j(s)$ lie in the set:
\[
\mathcal{K}_S:=\{\vphi\in \cE^1; \sup(\vphi-{\psi_0})= 0 \text{ and } \bfH(\vphi)\le (\gamma+n) s+C \}.
\]
This is a compact subset of the finite energy space $\cE^1$ by Theorem \ref{thm-BBEGZ} from \cite{BBEGZ}. So, by arguing as in \cite{BBJ15}, after passing to a subsequence, $\{\vphi_j(s)\}$ converges to a geodesic ray
$\Phi:=\{\vphi(s)\}_{s\ge 0}$ in $\cE^1$, uniformly for each compact time interval. $\{\vphi(s)\}$ satisfies 
\begin{equation}\label{eq-EPhi*}
\sup(\vphi(s)-\psi_0)=0, \quad \bfE(\vphi(s))=-s.
\end{equation}
By the lower semicontinuity of the Mabuchi functional under strong convergence of potentials, we have:
\begin{equation}\label{eq-Dvphiupb}
\bfD(\vphi(s))\le \bfM(\vphi(s))\le \gamma s+C  \text{ for } s\ge 0.
\end{equation}
For the first well-known inequality, see \cite[Lemma 4.4.(i)]{BBEGZ}.

\subsubsection{Convexity of Mabuchi energy}\label{sec-convex}

In this section we will prove the convexity of Mabuchi energy along geodesic segments connecting two metrics from $\hat{\cH}(X, \omega)$ (see \eqref{eq-hatcH}), which was used in the above construction of destabilizing geodesic ray.
We
fix a resolution of singularities $\mu: Y\rightarrow X$ such that $\mu$ is an isomorphism over $X^\reg$, $\mu^{-1}(X^\sing)=\sum_{k=1}^{g}E_k$ is a simple normal crossing divisor and that there exist $\theta_k\in \bQ_{>0}$ for $k=1,\dots, g$ such that $E_\theta:=\sum_{k=1}^g \theta_k E_k$ satisfies $P:=P_\theta=\mu^*L-E_\theta$ is an ample $\bQ$-divisor over $Y$. There is always such a resolution which is obtained by embedding $X$ into an smooth ambient spaces and taking the strict transform of $X$ under a sequence of blowups along smooth centers (see \cite{Kol07}). We can then choose and fix a smooth Hermitian metric $\psi_P$ on $P$ such that $\sddb \psi_P>0$. 
Let $D'=\mu_*^{-1}D$ be the strict transform of $D$ under $\mu$. Then we can write:
\begin{eqnarray}
-K_Y&=&\mu^*(-(K_X+D))+\mu_*^{-1}D+\sum_k b_k E_k\nonumber \\
&=&\frac{1}{1+\epsilon}\left(\mu^*L+\epsilon(\mu^*L-E_\theta\right)+D'+\sum_k \left(b_k+\frac{\epsilon}{1+\epsilon}\theta_k\right) E_k\nonumber \\
&=&\frac{1}{1+\epsilon}L_\epsilon+D'+B'_\epsilon=\frac{1}{1+\epsilon}L_\epsilon+B_\epsilon, \label{eq-KYdec}
\end{eqnarray} 
where we have set:
\begin{equation}\label{eq-B'ep}
L_\epsilon=\mu^*L+\epsilon P, \quad 
B':=B'_\epsilon=\sum_{k} \left(b_k+\frac{\epsilon}{1+\epsilon}\theta_k\right)E_k, \quad B_\epsilon=D'+B'_\epsilon.
\end{equation}
Note $L_\epsilon$ is a perturbation of the degenerate class $\mu^*L$, which is directly related to the above decomposition of $-K_Y$.

\begin{rem}
The above construction of the perturbation $L_\epsilon$ is a natural way to get ample $\bQ$-divisors on $Y$. Such choice will also be convenient for obtaining uniform K-stability on $(Y, B_\epsilon)$ in section \ref{sec-uniDing} and getting uniform estimates in our perturbation process such as in the proof of Proposition \ref{prop-limLep}.
\end{rem}

\begin{prop}\label{prop-Mabconv}
Assume $\vphi(0), \vphi(1)\in \hat{\mcH}(X, \omega)$. Let $\Phi=\{\vphi(s), s\in [0,1]\}$ be the geodesic joining $\vphi(0)$ and $\vphi(1)$ as constructed in Proposition \ref{prop-geod}. Then for any $s\in [0,1]$, we have the inequality:
\begin{equation} \label{eq-weakconvex}
\bfM_{\psi_0}(\vphi(s))\le (1-s)\bfM_{\psi_0}(\vphi_0)+s \bfM_{\psi_0}(\vphi_1).
\end{equation} 
\end{prop}
The rest of this subsection is devoted to proving this proposition. Fix a smooth volume form $\Omega$ on $Y$. Recall that we can write the Mabuchi energy on $\PSH_\bd(L)$ in the following form (see \eqref{eq-Mabuchi1}):
\begin{eqnarray*}
\bfM_{\psi_0}(\vphi)&=&\bfH(\vphi)-({\bf I}-\bfJ)_{\psi_0}(\vphi) \\
&=&\int_X \log\frac{|s_D|^2(\sddb\vphi)^n}{e^{-\psi_0}}(\sddb \vphi)^n-\bfE_{\psi_0}(\vphi)+\int_X (\vphi-\psi_0)(\sddb \vphi)^n\\
&=&\int_Y \log \frac{(\sddb\vphi)^n}{\Omega}(\sddb\vphi)^n-\int_Y\log \frac{e^{-\psi_0}}{|s_B|^2 \Omega}(\sddb\vphi)^n\\
&&\hskip 5cm +\int_Y (\vphi-\psi_0)(\sddb\vphi)^n-\bfE_{\psi_0}(\vphi)\\
&=:&\bfH_\Omega((\sddb\vphi)^n)-\bfT((\sddb\vphi)^n)+\int_Y (\vphi-\psi_0)(\sddb\vphi)^n-\bfE_{\psi_0}(\vphi),
\end{eqnarray*}
where $B=\mu_*^{-1}D+B'_0=D'+\sum_k b_k E_k$.
Now we want to compare this with the Mabuchi energy on $\PSH_\bd(L_\epsilon)$ for the effective pair $(Y, D')$. For any $\vphi\in \PSH_\bd(L_\epsilon)$, set
\begin{equation}
\bfM_{\psi_\epsilon}(\vphi)=\int_Y \log\frac{(\sddb\vphi)^n}{\Omega}(\sddb\vphi)^n-\bfE^{Ric(\Omega)}_{\psi_\epsilon}(\vphi)+\bfE^{D'}_{\psi_\epsilon|_{D'}}(\vphi|_{D'})+n\bar{S}_\epsilon\bfE_{\psi_\epsilon}(\vphi),
\end{equation}
where
\begin{equation}
\bar{S}_\epsilon:=\frac{-(K_Y+D')\cdot L_\epsilon^{\cdot n-1}}{L_\epsilon^{\cdot n}}=\frac{\left(\frac{1}{1+\epsilon}(L_0+\epsilon P)+B'_\epsilon\right)\cdot (L_0+\epsilon P)^{\cdot n}}{(L_0+\epsilon P)^{\cdot n-1}}
\end{equation}
converges to $1$ as $\epsilon\rightarrow 0$ (note that $B'_\epsilon$ is exceptional).
Moreover, for any $\eta$ a smooth (1,1)-form and for any $\bQ$-Weil divisor $Q$ on $Y$, we used the definition of the following functionals: 
\begin{eqnarray*}
\bfE^{\eta}_{\psi_\epsilon}(\vphi)
&=&\frac{1}{n}\int_Y (\vphi-\psi_\epsilon) \sum_{k=0}^{n-1}\eta \wedge (\sddb\psi_\epsilon)^k\wedge (\sddb\vphi)^{n-1-k},\\
\bfE^Q_{\psi_\epsilon|_Q}(\vphi|_Q)
&=&\frac{1}{n}\sum_{k=0}^{n-1} \int_Q (\vphi-\psi_\epsilon) (\sddb\psi_\epsilon)^{n-1-k}\wedge (\sddb \vphi)^k.
\end{eqnarray*}
It is an easy exercise to show that there exists $C_\epsilon\in \bR$ such that the following identity holds true:
\begin{eqnarray}
\bfM_\epsilon(\vphi)+C_\epsilon &=&\int_Y \log\frac{|s_B|^2(\sddb \vphi)^n}{e^{-\psi_0}}(\sddb\vphi)^n\nonumber\\
&&\hskip 5mm-(n+1-n \bar{S}_\epsilon)\bfE_{\psi_\epsilon}(\vphi)+\int_Y(\vphi-\psi_\epsilon)(\sddb\vphi)^n\nonumber \\
&&\hskip 5mm+n \epsilon \bfE^{\omega_P}_{\psi_\epsilon}(\vphi)-n \bfE^{B'_\epsilon}_{\psi_\epsilon|_{B'_\epsilon}}(\vphi|_{B'_\epsilon}). \label{eq-Mep2}
\end{eqnarray}
For simplicity of notations, we denote the right-hand-side of the above identity as:
\begin{eqnarray}\label{eq-hMep}
\hat{\bfM}_\epsilon(\vphi)&=&\bfH_{\nu_0}((\sddb\vphi)^n)+F_\epsilon(\vphi),
\end{eqnarray}
where $\nu_0=\frac{e^{-\psi_0}}{|s_B|^2}$ and the entropy part is given by
\begin{equation}\label{eq-entropynu0}
\bfH_{\nu_0}((\sddb\vphi)^n):=\int_Y \log \frac{(\sddb\vphi)^n}{\nu_0}(\sddb\vphi)^n.
\end{equation}

One sees immediately that 
\begin{equation}\label{eq-hatM0}
\hat{\bfM}_0(\vphi)=\bfM(\vphi) \text{ for all } \vphi\in \PSH_\bd(L).
\end{equation}

Fix any $\vphi\in \hat{\mcH}(\omega_0)$ with $\int_Y (\vphi-\psi_0)\omega_1^n=0$. Set $u_0=\vphi-\psi_0\in \PSH_\bd(\omega)$ and $g:=g(\vphi)=\frac{(\sddb\vphi)^n}{\Omega}\in C^\infty(Y)$ (see \eqref{eq-hatcH}).
Solve the Monge-Amp\`{e}re equation:
\begin{equation}\label{eq-CMAep}
(\sddb (\psi_0+\epsilon \psi_P)+\sddb u_\epsilon)^n= d_{\epsilon}\cdot g \Omega, \quad \int_Y u_\epsilon \omega_1^n=0,
\end{equation}
where $d_\epsilon=(2\pi)^n L_\epsilon^{\cdot n}/(\int_Y g\Omega)$ convergs to $d_0$ as $\epsilon\rightarrow 0$. Then, since $\omega_\epsilon=\sddb(\psi_0+\epsilon\psi_P)$ is K\"{a}hler, there exists a solution $u_\epsilon \in C^\infty(Y)$ by \cite{Yau78}. 
Denote by $E=\sum_i E_i$ the (reduced, simple normal crossing) exceptional divisor and fix a smooth Hermitian metric on the line bundle associated to the divisor $E$. Then we have the following uniform estimates:
\begin{prop}[see {\cite[Theorem 3.5]{EGZ09}, \cite[Theorem 6.1]{DP10}}]\label{prop-partialest}
There exist positive constants $\alpha>0$ and $C>0$ independent of $\epsilon$ such that, for any $0<\epsilon\le  1$, we have the uniform estimates:
\begin{equation}\label{eq-partialest}
\|u_\epsilon\|_{C^0}\le C_0, \quad |\partial\bar{\partial} u_\epsilon|_{\omega_1}\le \frac{C}{|s_E|_{h_E}^{2\alpha}},
\end{equation}
where $\omega_1$ is a K\"{a}hler metric on $Y$ (see \eqref{eq-omegaep}).
Moreover $u_\epsilon$ converges to $u_0$ in $L^1(\omega_1^n)$ over $Y$ and locally uniformly over $Y\setminus E$ in smooth topology.
\end{prop}
These are well-known estimates derived in the study of degenerate complex Monge-Amp\`{e}re equations. 
For the reader's convenience, we sketch their proofs and refer to \cite[Proof of Theorem 3.5]{EGZ09} or \cite[Proof of Theorem 6.1]{DP10} for more details. 
\begin{proof}[Sketch of proof of Proposition \ref{prop-partialest}]

The uniform $C^0$-estimate in \eqref{eq-partialest} follows directly from \cite[Theorem 2.1]{EGZ09} which depends on Ko\l odziej's $L^\infty$-estimate from \cite{Kol98}. The second (partial $C^2$)-estimate in \eqref{eq-partialest} can be obtained using the same arguments as in \cite[Proof of Theorem 3.5]{EGZ09} which depend on Tsuji's trick and Yau's estimate (\cite{Yau78}) as follows. Recall that $P=\mu^*L-E_\theta$ is equipped with a smooth Hermitian metric $\psi_P$ satisfying $\eta:=\sddb\psi_P>0$. Set $\psi_\theta=\psi_0-\psi_P$ to be a smooth Hermitian metric on the $\bQ$-line bundle associated to $E_\theta=\sum_i \theta_i E_i$. Then over the locus $Y\setminus \cup_i E_i$ the equation \eqref{eq-CMAep} can be re-written as:
\begin{equation}
(\eta+\epsilon (\sddb \psi_P)+\sddb v_\epsilon)^n=F_\epsilon\cdot \eta^n
\end{equation}
where $v_\epsilon=u_\epsilon-\log (|s_{E_\theta}|^2e^{-\psi_\theta})$ with $s_{E_\theta}$ being the defining section of the divisor $E_\theta$, and $F_{\epsilon}=d_\epsilon \cdot g \Omega/\eta^n$.  
For any $\epsilon \in (0,1]$, denote by $\Delta$ (resp. $\Delta'$) the Laplace operator associated to the K\"{a}hler metric $\eta_\epsilon:=\eta+\epsilon \sddb\psi_P$ (resp. $\eta+\epsilon \sddb\psi_P+\sddb v_\epsilon$). 
 For $0<\epsilon\le 1$, the K\"{a}hler forms $\eta_\epsilon=\eta+\epsilon \sddb\psi_P$ and the smooth functions $F_\epsilon$ are both uniformly bounded in $C^\infty$-topology. Set $f=n+\Delta v_\epsilon$. By applying Yau's calculation in \cite[(2.18)-(2.22)]{Yau78}, we know that there exists $\Lambda=\Lambda(n, \eta, \psi_P)>0$ such that if $\lambda \ge \Lambda$, there is a constant $C=C(n, \eta, \psi_P)>0$ that does not depend on $\epsilon\in (0,1]$ satisfying the following estimate over $Y\setminus \cup_i E_i$:
\begin{eqnarray*}
e^{\lambda v_\epsilon} \Delta'(e^{-\lambda v_\epsilon}f) \ge -C-C f+C f^{\frac{n}{n-1}}.
\end{eqnarray*}
Because $e^{-\lambda v_\epsilon} f$ vanishes along $E$, it obtains a maximum at $p_\epsilon\in Y\setminus \cup_i E_i$. By the maximum principle, we get:
\begin{equation*}
0\ge -C-C f+C f^{\frac{n}{n-1}} \quad \text{ at point } p_\epsilon.
\end{equation*}
So we get $(n+\Delta v_\epsilon)(p_\epsilon) \le C'$ for some constant $C'$ independent of $\epsilon$. This gives the following estimate over $Y\setminus \cup_i E_i$: 
\begin{eqnarray*}
n+\Delta u_\epsilon + \tr_{\eta_\epsilon}(\sddb \psi_\theta)&=&n+\Delta v_\epsilon \le C' e^{\lambda (v_\epsilon-v_\epsilon(p_\epsilon))}\\
&=&C' e^{\lambda (u_\epsilon-u_\epsilon(p_\epsilon))} \frac{(|s_{E_\theta}|^2e^{-\psi_\theta})^\lambda(p_\epsilon)}{(|s_{E_\theta}|^2 e^{-\psi_\theta})^\lambda}\\
&\le& C'' \frac{1}{(|s_{E_\theta}|^2e^{-\psi_\theta})^\lambda} 
\end{eqnarray*}
where we used the uniform estimate $\|u_\epsilon\|_{C^0}\le C_0$ and $C''=C''(C_0, \Lambda, \psi_\theta)>0$ is a constant independent of $\epsilon$. 
Now because $\eta_\epsilon$ is uniformly comparable to $\omega_1$ and $\tr_{\eta_\epsilon}(\sddb\psi_\theta)$ is uniformly bounded with respect to $\epsilon\in (0,1]$, it is easy to get the wanted estimate
with $\alpha=\max_i \{\lambda \theta_i\}$. 

Finally by Evans-Krylov's estimate, we get higher order local estimates over $Y\setminus \cup_i E_i$. This implies that $u_\epsilon$ converges to a bounded solution to $((\sddb \psi_0)+\sddb u)^n=(\sddb\vphi)^n$ locally uniformly in smooth topology over $Y\setminus \cup_i E_i$, and globally in $L^1(\omega_1^n)$.  
\end{proof}

\begin{lem}\label{lem-endEconv}
Set $\vphi_\epsilon=\psi_\epsilon+u_\epsilon$. Then we have the convergence:
\begin{equation}\label{eq-Econvend}
\lim_{\epsilon\rightarrow 0}\bfE_{\psi_\epsilon}(\vphi_\epsilon)=\bfE_{\psi_0}(\vphi).
\end{equation}
\end{lem}
\begin{proof}
First note that $\PSH(\omega)\subset \PSH(\omega_\epsilon)$ and we can write:
\begin{eqnarray}
\bfE_{\psi_\epsilon}(\vphi_\epsilon)-\bfE_{\psi_0}(\vphi)&=&\left(\bfE_{\psi_\epsilon}(\psi_\epsilon+u_\epsilon)-\bfE_{\psi_\epsilon}(\psi_\epsilon+u_0)\right)+\left(\bfE_{\psi_\epsilon}(\psi_\epsilon+u_0)-\bfE_{\psi_0}(\psi_0+u_0)\right)\nonumber \\
&=:&{\bf P}_1+{\bf P_2}.  \label{eq-P1P2}
\end{eqnarray}
We will prove that both ${\bf P}_1$ and ${\bf P}_2$ converge to 0 as $\epsilon\rightarrow 0$. 
Note that $\bfE$ satisfies the co-cycle condition. So we have:
\begin{eqnarray*}
{\bf P}_1&=&\bfE_{\psi_\epsilon+u_0}(\psi_\epsilon+u_\epsilon)\\
&=&\bfE_{\psi_\epsilon+u_0}(\psi_\epsilon+u_\epsilon)-\int_Y(u_\epsilon-u_0)(\sddb(\psi_\epsilon+u_0))^n+\int_Y(u_\epsilon-u_0)((\sddb(\psi_\epsilon+u_0)^n)\\
&=&
-\bfJ_{\psi_\epsilon+u_0}(\psi_\epsilon+u_\epsilon)+\int_Y (u_\epsilon-u_0)\left(\sddb(\psi_\epsilon+u_0)^n)-(\sddb(\psi_0+u_0))^n\right)\\
&&\hskip 3cm+\int_Y (u_\epsilon-u_0)(\sddb(\psi_0+u_0))^n\\
&=&-\bfJ_{\psi_\epsilon+u_0}(\psi_\epsilon+u_\epsilon)+{\bf P}'_1+{\bf P}''_1.
\end{eqnarray*}
Because $u_\epsilon$ and $u_0$ are uniformly bounded with respect to $\epsilon$, and (recall that $\psi_\epsilon=\psi+\epsilon \psi_P$)
\begin{eqnarray}
&&(\sddb (\psi_\epsilon+u_0))^n-(\sddb (\psi_0+u_0))^n\nonumber \\
&=&\epsilon \sddb \psi_P\wedge \sum_{k=0}^{n-1}(\sddb (\psi_\epsilon+u_0))^k\wedge (\sddb(\psi_0+u_0))^{n-1-k} \label{eq-epmeasure}
\end{eqnarray} 
is a positive measure that converges to 0 weakly as $\epsilon\rightarrow 0$, we see that ${\bf P}'_1$ converges to $0$ as $\epsilon\rightarrow 0$. Moreover, because $u_\epsilon$ converges to $u_0$ in $L^1(\omega_1^n)$, ${\bf P}''_1$ converges to 0 as $\epsilon \rightarrow 0$. So we just need to consider the first term.
By inequality \eqref{eq-IJineq}, we have:
\begin{eqnarray*}
\bfJ_{\psi_\epsilon+u_0}(\psi_\epsilon+u_\epsilon)&\le& \frac{n}{n+1}\bfI_{\psi_\epsilon+u_0}(\psi_\epsilon+u_\epsilon)\\
&=&\frac{n}{n+1}\int_Y\left((u_\epsilon-u_0)((\sddb(\psi_\epsilon+u_0))^n-(\sddb(\psi_\epsilon+u_\epsilon))^n\right)
\end{eqnarray*}
When $\epsilon\rightarrow 0$, the integral $\int_Y (u_\epsilon-u_0)(\sddb(\psi_\epsilon+u_0))^n={\bf P}'_1+{\bf P}''_2$ converges to 0 as before. For the other integral, we use the method in \cite[Proof of Theorem 2.17]{BBEGZ} as follows. By H\"{o}lder-Young inequality from \cite[Proposition 2.15]{BBEGZ}:
\begin{eqnarray}
\int_Y (u_\epsilon-u_0)((\sddb (\psi_\epsilon+u_\epsilon))^n&\le& \|u_\epsilon-u_0\|_{L^{\chi^*}(\Omega)}\cdot \left\|\frac{(\sddb(\psi_\epsilon+u_\epsilon))^n}{\Omega}\right\|_{L^{\chi}(\Omega)}\nonumber\\
&=&\|u_\epsilon-u_0\|_{L^{\chi^*}(\Omega)}\cdot \|d_\epsilon g\|_{L^{\chi}(\Omega)}\nonumber \\
&\le& C\|u_\epsilon-u_0\|_{L^{\chi^*}(\Omega)}, \label{eq-Lchi}
\end{eqnarray}
where $\chi(s)=(s+1)\log(s+1)-s$, $\chi^*(s)=e^s-s-1$, and the norm $L^\chi(\nu)$ (and similarly $L^{\chi^*}(\nu)$) for a measure $\nu$ and weight function $\chi$ is defined as:
\begin{equation}
\|f\|_{L^\chi(\nu)}:=\inf\left\{\lambda>0, \int_Y \chi\left(\lambda^{-1}|f|\right)\le 1\right\}.
\end{equation}
To show that $\|u_\epsilon-u_0\|_{L^{\chi^*}(\Omega)}$ converges to 0,
by uing the inequality $\chi^*(t)\le t e^t$, it is then enough to show that, for any given $\lambda>0$,
\begin{equation}\label{eq-uexpu}
\lim_{\epsilon\rightarrow 0}\int_Y|u_\epsilon-u_0|\exp\left(\lambda|u_\epsilon-u_0|\right)\Omega=0. 
\end{equation}
By \cite[Proposition 1.4]{BBEGZ}, $\int_Y e^{-2\lambda u_\epsilon}\Omega$ and $\int_Y e^{-2\lambda u_0}\Omega$ are uniformly bounded for some constant $B$ independent of $\epsilon$.
Then \eqref{eq-uexpu} follows from the standard H\"{o}lder's inequality and the convergence of $u_\epsilon\rightarrow u_0$ in $L^2(\nu)$ for any tame measure $\nu$ (see \cite[Proposition 1.4]{BBEGZ}).

Finally, the ${\bf P}_2$ part in \eqref{eq-P1P2} converges to 0 by using the formula of $\bfE$:
\begin{eqnarray*}
{\bf P}_2&=&\int_Y u_0\left(\sum_{k=0}^n (\sddb (\psi_0+\epsilon\psi_P+u_0))^k\wedge (\sddb (\psi_0+\epsilon \psi_P))^{n-k}\right.\\
&&\hskip 3cm -\left.\sum_{k=0}^n(\sddb(\psi_0+u_0))^k\wedge (\sddb\psi_0)^{n-k}\right)
\end{eqnarray*}
and the fact the positive measure in the bracket converges to 0 as $\epsilon\rightarrow 0$.

\end{proof}

From now on, we fix $\vphi(0), \vphi(1)\in \hat{\mcH}(\omega)$ such that $u(0)=\vphi(0)-\psi$ and $u(1)=\vphi(1)-\psi$ satisfy $\int_X u(i)\omega_1^n=0, i=1,2$. 
Set $g_i=\frac{(\sddb\vphi(i))^n}{\Omega}\in C^\infty(Y)$.
We solve the same equation as in \eqref{eq-CMAep} to get approximations $\vphi_\epsilon(0)=\psi_\epsilon+u_\epsilon(0)$ and $\vphi_\epsilon(1)=\psi_\epsilon+u_\epsilon(1)$:
\begin{equation}\label{eq-CMAepi}
(\sddb (\psi_\epsilon+\sddb u_\epsilon(i))^n= d_{\epsilon,i}\cdot g_i \Omega, \quad \int_Y u_\epsilon(i) \omega_1^n=0,
\end{equation}
where $d_{\epsilon,i}=(2\pi)^n L_\epsilon^{\cdot n}/(\int_Y g_i \Omega)$.

Let $\Phi_\epsilon=\{\vphi_\epsilon(s)\}_{s\in [0,1]}$ be the geodesic segment connecting Hermitian metrics $\vphi_\epsilon(0), \vphi_\epsilon(1)\in \mcH(\omega_\epsilon)$. It is known that $\Phi_\epsilon \in C^{1,1}(Y\times \bD_{[0,1]})$ (see \cite{Che00, CTW17}). Moreover by \cite{BB17} $\hat{\bfM}_\epsilon(\vphi_\epsilon(s))$ is convex in $s\in [0, 1]$.

Set a sequence $\epsilon_k=2^{-k}$ which converges to 0 as $k\rightarrow+\infty$
and consider the geodesic segment $\vphi_{\epsilon_k}(s)$ joining $\vphi_{\epsilon_k}(0)=\vphi_{2^{-k}}(0)$ and $\vphi_{\epsilon_k}(1)=\vphi_{2^{-k}}(1)$. 

\begin{lem}\label{lem-geodconv}
$\vphi_{\epsilon_k}(s)-\psi_{\epsilon_k}$ are uniformly bounded with respect to both $k$ and $s$. 
As $k\rightarrow +\infty$, $\vphi_{\epsilon_k}(s)$ subsequentially converges pointwisely to the weak geodesic segment $\vphi(s)$ connecting $\vphi(0)$ and $\vphi(1)$. 
Moreover, 
\begin{equation}
\bfE_{\psi_{\epsilon_k}}(\vphi_{\epsilon_k}(s))=\bfE_{\psi_0}(\vphi(s))=s.
\end{equation}
\end{lem}
\begin{proof}
Set $A=\|u_\epsilon(1)-u_\epsilon(0)\|_{L^\infty}$. Then by the definition of geodesics using envelopes (see \eqref{eq-envelope}) and the maximal principle, we have the estimate (see \cite[proof of Proposition 1.4]{DNG18}):
$u_\epsilon(0)-A s\le u_\epsilon(s)\le u_\epsilon(0)+A s$ for any $s\in [0,1]$ which gives the estimate:
\begin{equation}
\|u_\epsilon(s)-u_\epsilon(0)\|\le \|u_\epsilon(1)-u_\epsilon(0)\|.
\end{equation}
So the first statement follows from the uniform boundedness of $u_\epsilon(0)$ and $u_\epsilon(1)$ with respect to $s$.

Recall that the entropy part of the Mabuchi energy is given by
\begin{equation}
\bfH_{\nu_0}((\sddb\vphi_\epsilon(s))^n):=\int_Y \log \frac{(\sddb\vphi_\epsilon(s))^n}{\nu_0}(\sddb\vphi_\epsilon(s))^n.
\end{equation}
where $\nu_0=\frac{e^{-\psi_0}}{|s_B|^2}$.
For $i=0,1$,
it is easy to see that $\bfH_{\nu_0}((\sddb\vphi_\epsilon(i))^n)$ are uniformly bounded with respect to $\epsilon$ by using the equation \eqref{eq-CMAepi}.

Recall the following formula from \eqref{eq-Mep2}-\eqref{eq-hMep}
$$\hat{\bfM}_\epsilon(\vphi_\epsilon(s))=\bfH_{\nu_0}(((\sddb\vphi_\epsilon(s))^n)+F_\epsilon(\vphi_\epsilon(s))$$
where 
\begin{eqnarray*}
F_\epsilon(\vphi(s))&=&-(n+1-n \bar{S}_\epsilon)\bfE_{\psi_\epsilon}(\vphi)+\int_Y(\vphi-\psi_\epsilon)(\sddb\vphi)^n\nonumber \\
&&\hskip 4cm +n \epsilon \bfE^{\omega_P}_{\psi_\epsilon}(\vphi)-n \bfE^{B'}_{\psi_\epsilon|_{B'}}(\vphi|_{B'}). \label{eq-Mep2}
\end{eqnarray*}
By the above discussion, it is easy to see that $\hat{\bfM}_\epsilon(\vphi_\epsilon(0))$ and $\hat{\bfM}_\epsilon(\vphi_\epsilon(1))$ are uniformly bounded.
By the convexity $s\mapsto \hat{\bfM}_{\epsilon}(\vphi_\epsilon(s))$, we know that $\hat{\bfM}_\epsilon(\vphi(s))$ is uniformly bounded with respect to $s\in [0,1]$.
Because $u_\epsilon(s)=\vphi_\epsilon(s)-\psi_\epsilon$ is uniformly bounded with respect to $\epsilon$ and $s$, we know that $F_{\epsilon}(\vphi_\epsilon(s))$ is uniformly bounded with respect to $\epsilon$ and $s$.  As a consequence we also get that the entropy 
$\bfH_{\nu_0}((\sddb\vphi_\epsilon(s))^n$ are uniformly bounded.

By the weak compactness of uniformly bounded quasi-psh functions, we know that   
$\vphi_{\epsilon_k}(s)=\vphi_{2^{-k}}(s)\rightarrow \tilde{\vphi}(s)$ in $L^1(\omega_1^n)$, after passing to a subsequence. Then we can prove that $\bfE_{\psi_{\epsilon_k}}(\vphi_{\epsilon_k}(s))\rightarrow \bfE_{\psi_0}(\tilde{\vphi}(s))$ by the same argument as in the proof of the convergence \eqref{eq-Econvend}.

Because $\vphi_\epsilon(s)$ is a geodesic with uniformly bounded potentials, by using \cite[Proposition 1.4.(ii)]{DNG18} there exist $C>0$ independent of $\epsilon$ and $s$ such that:
\begin{equation}
\|\vphi_\epsilon(s_1)-\vphi_\epsilon(s_2)\|_{L^\infty}\le C |s_1-s_2|.
\end{equation}
So we see that $[0, 1]\rightarrow \vphi_\epsilon(s)$ is equicontinuous in $L^1(\omega_1^n)$. By Arzel\`{a}-Ascoli, $\vphi_{2^{-k}}(s)$ subsequentially converges uniformly in $L^1(\omega_1^n)$ topology to some $\omega_1$-psh-path $\tilde{\vphi}(s)$ joining $\vphi(0)$ and $\vphi(1)$ (see \cite[Proposition 1.4]{BBJ18}). 
In particular, the positive currents $p_1^*\omega_1+\sddb_{z,s} \vphi_\epsilon$ converge to a positive current $p_1^*\omega_1+\sddb_{z,s}\tilde{\vphi}$. As a consequence, the positive currents $p_1^*\omega_\epsilon+\sddb_{z,s}\vphi_\epsilon$ converge to a positive current $p_1^*\omega_0+\sddb_{z,s}\tilde{\vphi}$. Thus $\tilde{\vphi}$ is also a $\omega_0$-psh path.

Because $\bfE_{\psi_0}(\tilde{\vphi}(s))$, being the pointwise limit of the affine function $\bfE_{\psi_{2^{-k}}}(\vphi_{2^{-k}}(s))$, is also affine in $s$, by \cite[Corollary 1.8]{BBJ18} which is also true in the current singular case, we know that $\tilde{\vphi}(s)$ is nothing but the geodesic joining $\vphi(0)$ and $\vphi(1)$.

\end{proof}
\begin{rem}
The above proof uses the boundedness of entropy to control the convergence of geodesic segments.
One should also be able to adapt \cite[Proof of Proposition 4.3]{BDL17} to prove the above convergence results.
\end{rem}

\begin{lem}\label{lem-conv3}
There is a sequence $\epsilon_k\rightarrow 0$ such that for any $s\in [0,1]$
we have the convergence:
\begin{equation}\label{eq-weakconv1}
\lim_{\epsilon_k\rightarrow 0}\int_Y (\vphi_{\epsilon_k}(s)-\psi_{\epsilon_k}) (\sddb \vphi_{\epsilon_k}(s))^n=\int_Y (\vphi(s)-\psi_0) (\sddb\vphi(s))^n.
\end{equation}

\end{lem}
\begin{proof}
With the above notations, $u_\epsilon=u_\epsilon(s)=\vphi_\epsilon(s)-\psi_\epsilon$. Then we can write:
\begin{eqnarray*}
&&\int_Y u_\epsilon(\sddb\vphi_\epsilon)^n-u_0 (\sddb\vphi_0)^n=\int_Y (u_\epsilon-u_0)(\sddb\vphi_\epsilon)^n+\\
&&\hskip 2cm+\int_Y u_0((\sddb\vphi_\epsilon)^n-(\sddb(\psi_\epsilon+u_0))^n)\\
&&\hskip 2cm+\int_Y u_0((\sddb(\psi_\epsilon+u_0))^n-(\sddb(\psi_0+u_0))^n).
\end{eqnarray*}
As in the proof of Lemma \ref{lem-endEconv} (see \eqref{eq-Lchi}-\eqref{eq-uexpu}), letting $\epsilon=\epsilon_k=2^{-k}$ we show that the first integral on the right-hand-side converges to 0 by using the uniform entropy bound and the weak convergence of $u_{\epsilon_k}$ to $u_0$ (by Lemma \ref{lem-geodconv}).  The last integral converges to 0 as $\epsilon_k\rightarrow 0$ because the positive measure in the bracket converges to 0 (see \eqref{eq-epmeasure}). To deal with the second integral on the right, 
we use the same proof as in \cite[Proof of Lemma 3.13]{BBGZ13} by 
setting
\begin{equation}\label{eq-alphap}
\alpha_p=\int_Y u_0 ((\sddb \vphi_\epsilon)^p\wedge (\sddb (\psi_\epsilon+u_0))^{n-p}.
\end{equation}
Using integration by parts and Schwarz inequality, we get:
\begin{eqnarray*}
\left|\alpha_{p+1}-\alpha_p\right|^2 &\le& C \int_Y\sqrt{-1}\partial (u_\epsilon-u_0) \wedge\bar{\partial} (u_\epsilon-u_0)\wedge ((\sddb \vphi_\epsilon)^p\wedge (\sddb (\psi_\epsilon+u_0)^{n-p-1} )\\
&=&C\int_Y (u_\epsilon-u_0)((\sddb(\psi_\epsilon+u_0))^n-(\sddb(\psi_\epsilon+u_\epsilon))^n)=C \cdot \bfI_{\psi_\epsilon+u_0}(\psi_\epsilon+u_\epsilon).
\end{eqnarray*}
Then we easy to get that:
\begin{eqnarray*}
|\alpha_n-\alpha_0|&\le &\left(\sum_{p=1}^n |\alpha_p-\alpha_{p-1}|^2\right)^{1/2} \le C \left(\bfI_{\psi_\epsilon+u_0}(\psi_\epsilon+u_\epsilon)\right)^{1/2}.
\end{eqnarray*}
Finally with $\epsilon=\epsilon_k=2^{-k}$, the same argument as in the proof of \ref{lem-endEconv} (see \eqref{eq-Lchi}-\eqref{eq-uexpu}) shows that $\bfI_{\psi_{\epsilon_k}+u_0}(\psi_{\epsilon_k}+u_{\epsilon_k})$ converges to 0, after passing to a subsequence.

\end{proof}

\begin{lem}
With the same sequence $\{\epsilon_k\}_k$ as above, for any $s\in [0,1]$, $(\sddb\vphi_{\epsilon_k}(s))^n$ converges to $(\sddb\vphi(s))^n$ weakly as $\epsilon_k\rightarrow 0$. 
\end{lem}
\begin{proof}
Since any continuous function on $X$ can be uniformly approximated by smooth functions, we just need to show that for any $v\in  C^\infty(X)$, there is a convergence:
\begin{equation*}
\lim_{\epsilon_k \rightarrow 0}\int_Y v(\sddb\vphi_{\epsilon_k}(s))^n=\int_Y v(\sddb\vphi(s))^n.
\end{equation*}
Because the reference $\omega$ is a restriction of positive K\"{a}hler form on $\bP^N$, we know that $v$ is $(m\omega)$-psh for some $m \gg 1$. So, after rescaling, it suffices to prove the convergence for $v\in \PSH(\omega)\cap C^\infty(X)$. 

Set $\vphi_0=\vphi$ and $u_0=\vphi_0-\psi$. We can then estimate the difference in a similar way to the proof of Lemma \ref{lem-conv3}. 
\begin{eqnarray*}
\int_Y v (\sddb\vphi_\epsilon)^n-v (\sddb\vphi_0)^n&=&\int_Y v((\sddb\vphi_\epsilon)^n-(\sddb(\psi_\epsilon+u_0))^n)\\
&&+
\int_Y v((\sddb(\psi_\epsilon+u_0))^n-(\sddb (\psi_0+u_0))^n).
\end{eqnarray*}
As before, it is easy to see that the second integral on the right converges to $0$ as $\epsilon\rightarrow 0$. The first integral on the right converges to $0$ by the same arguments as in the above proof: define the $\alpha_p$ similar to \eqref{eq-alphap}, integrate by parts, use Schwarz inequality/triangle inequality and finally use $\bfI_{\psi_\epsilon+u_0}(\vphi_\epsilon+u_\epsilon)\rightarrow 0$. We leave the details to the interested reader.
\end{proof}

\begin{lem}\label{lem-convEB'}
For any $s\in [0,1]$, we have the convergence:
\begin{equation}
\lim_{k\rightarrow +\infty}\bfE^{B'}_{\psi_{\epsilon_k}|_{B'}}(\vphi_{\epsilon_k}|_{B'})=0.
\end{equation}
\end{lem}
\begin{proof}
Because $u_{\epsilon_k}=\vphi_{\epsilon_k}-\psi_{\epsilon_k}$ is uniformly bounded independent of $\epsilon$, there exists $C>0$ such that 
\begin{eqnarray*}
&&\left|\frac{1}{n}\sum_{i=0}^{n-1} \int_{B'} (\vphi_{\epsilon_k}-\psi_{\epsilon_k}) (\sddb\psi_{\epsilon_k})^{n-1-i}\wedge (\sddb \vphi)^i\right|\le C \la L_{\epsilon_k}^{\cdot n-1}, B'\ra.
\end{eqnarray*}
The quantity on the right-hand-side converges to $0$ as $k\rightarrow +\infty$ because $B'$ is exceptional.
\end{proof}

\begin{proof}[Completion of the proof of Proposition \ref{prop-Mabconv}]
Under the convergence of geodesics $\vphi_\epsilon(s)$ to $\vphi(s)$ (from Lemma \ref{lem-geodconv}), the entropy part of the Mabuchi energy converges for the end points (i.e. $s=0,1$) and is lower semicontinuous in the middle (i.e. $s\in (0,1)$). 
The convergence of the other parts of the Mabuchi energy follows from Lemma \ref{lem-conv3} and Lemma \ref{lem-convEB'}. 

So we get that $\hat{\bfM}_{\epsilon_k}$ approximates $\bfM$ at end points: for $i=0, 1$: 
\begin{equation}\label{eq-endMconv}
\lim_{k\rightarrow +\infty} \hat{\bfM}_{\epsilon_k}(\vphi_{\epsilon_k}(i))=\hat{\bfM}_0(\vphi(i))=\bfM_{\psi_0}(\vphi(i)).
\end{equation}\label{eq-Mlsc}
For metrics in the interior of geodesics, we have:
\begin{equation}\label{eq-endconv}
\lim_{k\rightarrow+\infty}\hat{\bfM}_{\epsilon_k}(\vphi_{\epsilon_k}(s))\ge \hat{\bfM}_0(\vphi(s))=\bfM_{\psi_0}(\vphi(s)).
\end{equation}
By the convexity of $\hat{\bfM}_{\epsilon_k}(\vphi_\epsilon(s))$, we get:
\begin{equation}\label{eq-middle}
\hat{\bfM}_{\epsilon_k}(\vphi_{\epsilon_k}(s))\le \frac{S-s}{S}\hat{\bfM}_{\epsilon_k}(\vphi_\epsilon(0))+\frac{s}{S}\hat{\bfM}_{\epsilon_k}(\vphi_\epsilon(S)).
\end{equation}
The inequality \eqref{eq-weakconvex} follows by letting $\epsilon_k\rightarrow 0$ (see \eqref{eq-hatM0}).

\end{proof}
\begin{rem}\label{rem-smoothconvex}
In Proposition \ref{prop-Mabconv},
when $\vphi(0), \vphi(1)\in \mcH(X, \omega)$ (i.e. smooth K\"{a}hler metrics), R. Berman \cite{Berm19} showed us his earlier proof of the convexity along geodesics connecting smooth K\"{a}hler metrics. In this case, one can directly use the geodesics connecting $\vphi(0)+\epsilon \psi_P$ and $\vphi(1)+\epsilon \psi_P$ to approximate the geodesic connecting $\vphi(0), \vphi(1)$. Then the convergence of Mabuchi energy at the end points (\eqref{eq-endconv}) and the lower semicontinuity inequality \eqref{eq-middle} are easier to prove because the smoothness of the $\vphi(i)+\epsilon\psi_P$ and the easier convergence of geodesic segments. 

Under the assumption of uniform K-stability, the convexity of Mabuchi energy along geodesics connecting smooth K\"{a}hler metrics is enough for proving the properness of Mabuchi energy over $\mcH(X, \omega)$ using the same arguments in this paper (one still needs to use perturbation approach to deal with singularities). However as discussed in Remark \ref{rem-smoothproper}, there is still a difficulty to get KE metrics with this properness condition. In other words, we indeed need to prove the stronger convexity condition to get the properness over finite energy metrics. 
\end{rem}

\subsection{Step 2: Perturbed test configurations and perturbed $\bfE^\NA$ } \label{sec-pertTC}

As in section \ref{sec-convex}, we
fix a resolution of singularities $\mu: Y\rightarrow X$ such that $\mu$ is an isomorphism over $X^\reg$, $\mu^{-1}(X^\sing)=\sum_{k=1}^{g}E_k$ is a simple normal crossing divisor and that there exist $\theta_k\in \bQ_{>0}$ for $k=1,\dots, g$ such that $E_\theta=\sum_{k=1}^g \theta_k E_k$ satisfies $P:=P_\theta=\mu^*L-E_\theta$ is an ample $\bQ$-divisor over $Y$. Choose and fix a smooth Hermitian metric $\psi_P$ on $P$ such that $\sddb \psi_P>0$. 
For any $\epsilon\in \bQ_{>0}$, define a line bundle on $Y$ by 
\begin{equation}\label{eq-defhLep}
L_\epsilon:=(1+\epsilon)\mu^*L-\epsilon E_\theta=\mu^*L+\epsilon P. 
\end{equation}
Then $L_\epsilon$ is a positive $\bQ$-line bundle on $Y$. Define a smooth reference metric on $L_\epsilon$ by $\psi_\epsilon=\psi_0+\epsilon \psi_P$. 
In this section we will first construct a sequence of test configurations of $(Y, L_\epsilon)$ using the method from \cite{BBJ15}. 

Let $\Phi=\{\vphi(s)\}_{s\in [0,\infty)}$ be a geodesic ray in $\cE^1(X, L)$ satisfying:
\begin{equation}\label{eq-geocon}
\sup_X(\vphi(s)-\psi_0)=0, \quad \bfE_{\psi_0}(\vphi(s))=-s.
\end{equation}

Denote by $p'_i, i=1,2$ the projection of  $Y\times\bC$ to the two factors. Define a singular and a smooth Hermitian metric on $p'^*_1 L_\epsilon$ by 
\begin{equation}\label{eq-Phiep}
\Phi_\epsilon:=\bar{\mu}^*(\Phi)+\epsilon\; p'^*_1(\vphi_{P}), \quad \Psi_\epsilon:=p'^*_1(\mu^*\psi_0+\epsilon \psi_P).
\end{equation}
where $\bar{\mu}=\mu\times {\rm id}: Y\times\bC\rightarrow X\times \bC$.
Then $\sddb\Phi_\epsilon\ge 0$ and $\sddb \Psi_\epsilon\ge 0$. Consider the multiplier ideals $\mcJ(m\Phi_\epsilon)\subset \mcO_{Y\times\bC}$ which is defined over any open set $U\subset Y\times\bC$ as:
\begin{eqnarray*}
\mcJ(m\Phi_\epsilon)(U)&=&\mcJ(m\Phi)(U):=\{f\in \mcO_{Y\times\bC}(U); w(f)+A_{Y\times\bC}(w)-m w(\Phi)\ge 0 \\
&&\hskip 2cm \text{ for any divisorial valuation } w \text{ on } Y\times\bC\}.
\end{eqnarray*}
Here we identity $\Phi$ with its pulled-back metric $\bar{\mu}^*\Phi$ on $\bar{\mu}^*(L\times\bC)$. The first identity holds true because $\psi_P$ is a smooth Hermitian metric on $P$.
Denote $Y_\bC:=Y\times\bC$ and consider the following coherent sheaf:
\begin{eqnarray*}
\cF_{\epsilon,m}&:=&\cO_{Y_\bC}\left(p'^*_1(mL_\epsilon)\otimes \cJ(m \Phi_\epsilon)\right).
\end{eqnarray*} 
Fix a very ample line bundle $H'$ over $Y$. Then for any $1\le i\le n$ and $j\ge 0$, we can write: 
\begin{eqnarray*}
&&\cF_{\epsilon,m}\otimes p_1^*H'^{j-i}\\
&=&\cO_{Y_\bC}\left(p'^*_1(K_{Y}+m \mu^*L+(m\epsilon P-K_{Y}-(n+1)H')+(j+n+1-i)H')\otimes \cJ(m \bar{\mu}^*\Phi)\right).
\end{eqnarray*}
Because $P$ is positive, for $m\gg \epsilon^{-1}$ and sufficiently divisible, 
$m \epsilon P-K_{Y}-(n+1)H' $ is an ample line bundle on $Y$.
In this case, by Nadel vanishing theorem, for any $j\ge 1$, 
\begin{eqnarray*}
R^j(p'_2)_*(\cF_{\epsilon,m} \otimes p'^*_1H'^{-j})=0. 
\end{eqnarray*}
By the relative Castelnuovo-Mumford criterion, $\cF_{\epsilon,m}$ is $p'_2$-globally generated.

Let $\pi'_m: \cY_{\epsilon,m}\rightarrow Y_\bC$ denote the normalized blow-up of $Y\times \bC$ along $\cJ(m \Phi_\epsilon)=\cJ(m \bar{\mu}^*\Phi)$, with exceptional divisor $E_{\epsilon,m}$ and set 
\begin{equation}\label{eq-defcLepm}
\cL_{\epsilon,m}:=\pi'^*_m p_1'^*L_\epsilon-\frac{1}{m}E_{\epsilon,m}. 
\end{equation}
Then $(\cY_{\epsilon,m}, \cL_{\epsilon,m})$ is a normal semi-ample test configuration for $(Y, L_\epsilon)$ inducing a non-Archimedean metric $\Phi^\NA_{\epsilon,m}\in \cH^\NA(Y, L_\epsilon)$ and $\Phi^\NA_{\epsilon,m}\in \mcH^\NA(Y, L_\epsilon)$ given by:
\begin{equation}\label{eq-defphiepm}
\Phi^\NA_{\epsilon,m}(v)=-\frac{1}{m}G(v)(\cJ(m \bar{\mu}^*\Phi)) 
\end{equation}
for each divisorial valuation $v$ on $Y$. See Definition \ref{defn-TCNA}.

Let $\Phi_{\epsilon,m}$ be a locally bounded and positively curved Hermitian metric on $\mcL_{\epsilon,m}$.
By Demailly's regularization result (\cite[Proposition 3.1]{Dem92}), $\Phi_{\epsilon,m}$ is less singular then $\Phi_{\epsilon}$. By the monotonicity of $\bfE$ energy, we get:
\begin{eqnarray}\label{eq-ENAepmlb}
\bfE'^\infty_{\psi_\epsilon}(\Phi_{\epsilon,m}):=\lim_{s\rightarrow+\infty} \frac{\bfE_{\psi_\epsilon}(\vphi_{\epsilon,m}(s))}{s} \ge \lim_{s\rightarrow+\infty} \frac{\bfE_{\psi_\epsilon}(\vphi_{\epsilon}(s))}{s}=:\bfE'^\infty_{\psi_\epsilon}(\Phi_\epsilon).
\end{eqnarray}
The following key observation proves \eqref{eq-sklimEep}. 
\begin{prop}\label{prop-E'infconv}
With the above notations and assuming that $\Phi$ satisfies \eqref{eq-geocon}, the following convergence holds:
\begin{equation}\label{eq-limENAep}
\lim_{\epsilon \rightarrow 0} \bfE'^\infty_{\psi_\epsilon}(\Phi_\epsilon)=\lim_{s\rightarrow+\infty} \frac{\bfE_{\psi_0}(\vphi(s))}{s}=:\bfE'^\infty_{\psi_0}(\Phi).
\end{equation}
\end{prop}
\begin{proof}
Note that $\vphi_\epsilon(s)-\psi_\epsilon=\vphi+\epsilon \psi_P-(\psi_0+\epsilon \psi_P)=\vphi(s)-\psi_0$. So we get:
\begin{eqnarray*}
 \bfE_{\psi_0}(\vphi(s))&=&\frac{1}{n+1}\sum_i \int_X (\vphi(s)-\psi_0)(\sddb\vphi(s))^i\wedge (\sddb\psi_0)^{n-i}=:f_0(s)\\
&&\\
 \bfE_{\psi_\epsilon}(\vphi_\epsilon(s))&=&\frac{1}{n+1}\sum_{i=0}^n \int_X (\vphi_\epsilon(s)-\psi_\epsilon) (\sddb \vphi_\epsilon(s))^i\wedge (\sddb\psi_\epsilon)^{n-i}\\
&=&\frac{1}{n+1}\sum_i \int_X (\vphi(s)-\psi_0)(\sddb\vphi+\epsilon\sddb\psi_P)^i\\
&&\hskip 5cm \wedge (\sddb\psi_0+\epsilon\sddb\psi_P)^{n-i}\\
&=&f_0(s)+\int_X (\vphi(s)-\psi_0)T\\
&=:&f_\epsilon(s),
\end{eqnarray*}
where $T=T(\epsilon)$ is a positive $(n,n)$-current which approaches $0$ as $\epsilon\rightarrow 0$. 
Because the following hold true:
\begin{eqnarray*}
\sddb (f_\epsilon(s))&=&\int_X (\sddb \Phi_\epsilon)^{n+1}\ge 0, \\
\sddb (f_0(s))&=&\int_X (\sddb \Phi)^{n+1}=0,
\end{eqnarray*}
$f_\epsilon(s)$ is a convex funtion (by using a standard regularization argument) and $f_0(s)=-s$ is a linear function with respect to $s\in [0, +\infty)$.

Because $\vphi(s)-\psi_0\le 0$, we have
$f_\epsilon(s)\le f_0(s)$. So we get:
\begin{eqnarray}\label{eq-ENAeple}
\lim_{s\rightarrow +\infty}\frac{f_\epsilon(s)}{s}\le \lim_{s\rightarrow +\infty} \frac{f_0(s)}{s}=\lim_{s\rightarrow+\infty}\frac{\bfE_{\psi_0}(\vphi(s))}{s}=-1.
\end{eqnarray}
On the other hand, it follows from the above expressions of $f_\epsilon(s)$ that for any $s\ge 0$, $\lim_{\epsilon\rightarrow 0}f_\epsilon(s)=f_0(s)$.  
Fix $s_*>0$. By the convexity of $f_\epsilon(s)$, we have:
\begin{eqnarray*}
\lim_{s\rightarrow +\infty}\frac{f_\epsilon(s)}{s} \ge \frac{f_\epsilon(s_*)}{s_*}.
\end{eqnarray*}
Letting $\epsilon\rightarrow 0$, we get
\begin{eqnarray}\label{eq-ENAepge}
\lim_{s\rightarrow +\infty}\frac{f_\epsilon(s)}{s}\ge \frac{f_0(s_*)}{s_*}=-1.
\end{eqnarray}
Now \eqref{eq-limENAep} follows from \eqref{eq-ENAeple} and \eqref{eq-ENAepge}.
\end{proof}

\subsection{Step 3: Perturbed $\bL^\NA$ function}\label{sec-pertLNA}
Recall that we have the identity (see \eqref{eq-KYdec}):
\begin{eqnarray*}
K_Y+D'=\mu^*(K_X+D)+\sum_{k=1}^g a_k E_k=\mu^*(K_X+D)-\sum_{i=1}^{g_1} b_i E'_i+\sum_{j=g_1+1}^g a_jE''_j,
\end{eqnarray*}
where $D'=\mu^{-1}_* D$; for $i=1,\dots, g_1$, $E'_i=E_i$, $b_i=-a_i\in [0,1)$; for $j=g_1+1,\dots, g$, $a_j>0$ and $E''_j=E_j$. Denote by $\lceil a_j\rceil$ the round up of $a_j$ and $\{a_j\}=\lceil a_j\rceil-a_j\in [0,1)$. Then we re-write the above identity as:
\begin{eqnarray*}
-K_Y+\sum_j\lceil a_j\rceil E_j&=&\mu^*(-K_X-D)+D'+\sum_i b_iE'_i+\sum_j\{a_j\} E''_j\\
&=&\frac{1}{1+\epsilon}\left((1+\epsilon)\mu^*(-K_X-D)-\epsilon \sum_k \theta_k E_k\right)+D'\\
&&\hskip 3cm +\sum_i (b_i+\frac{\epsilon}{1+\epsilon}\theta_i)E'_i+\sum_j (\{a_j\}+\frac{\epsilon}{1+\epsilon} \theta_j)E''_j\\
&=&\frac{1}{1+\epsilon}(\mu^*(-K_X-D))+\epsilon P)+\Delta_\epsilon,
\end{eqnarray*}
where $P=\mu^*(-K_X-D)-\sum_k \theta_k E_k$ and 
\begin{eqnarray*}
\Delta_\epsilon=D'+\sum_i b_i E'_i+\sum_j\{a_j\} E''_j+\frac{\epsilon}{1+\epsilon} \sum_k\theta_k E_k=:\Delta_0+\frac{\epsilon}{1+\epsilon} E_\theta.
\end{eqnarray*}
Note that $\Delta_\epsilon$ is a simple normal crossing divisor with $\lfloor \Delta_\epsilon\rfloor=0$.
For simplicity of notations, we set $G:=\sum_j \lceil a_j\rceil E''_j$ and $B_\epsilon=\Delta_\epsilon-G$. Then we have:
\begin{equation}\label{eq-KYG}
-K_Y=\frac{1}{1+\epsilon}( \mu^*(-K_X-D)+\epsilon P)+\Delta_\epsilon-G=:\frac{1}{1+\epsilon} L_\epsilon+B_\epsilon.
\end{equation}

Consider the Ding energy \eqref{eq-DB} associated to this decomposition. Denote $V_\epsilon=(2\pi)^n L_\epsilon^{\cdot n}$. For any $\vphi_\epsilon\in \mcE^1(Y, L_\epsilon)$, denote:
\begin{eqnarray*}
\bfD_{\epsilon}(\vphi_\epsilon)&=&-\bfE_{\psi_\epsilon}(\vphi_\epsilon)+\bL_{(Y, B_\epsilon)}(\vphi_\epsilon)
\end{eqnarray*}
where $\psi_\epsilon=\psi_0+\epsilon \psi_P$, $B_\epsilon=\Delta_\epsilon-G$ and with $\lambda=\frac{1}{1+\epsilon}$ in \eqref{eq-LB}-\eqref{eq-DB}
\begin{equation}
\bL_\epsilon(\vphi_\epsilon):=\bL_{(Y, B_\epsilon)}(\vphi_\epsilon)=-V_\epsilon(1+\epsilon)\cdot \log\left( \int_Y e^{-\frac{\vphi_\epsilon}{1+\epsilon}}\frac{|s_{G}|^2}{|s_{\Delta_\epsilon}|^2}\right).
\end{equation}
The proof of the following lemma is similar to an argument from \cite[Proof of Theorem 5.1]{BBEGZ} ($\epsilon=0$ case).
\begin{lem}\label{lem-conv}
With the above notations, let $\epsilon$ be sufficiently small such that $\lfloor \Delta_\epsilon\rfloor=0$.
Assume that $\Phi_\epsilon=\{\vphi_\epsilon(s)\}$ is a psh ray in $\cE^1(Y, L_\epsilon)$. Then 
$\bL_{(Y, B_\epsilon)}(\vphi_\epsilon(s))$ is convex in $s=\log|t|^{-1}$.
\end{lem}
\begin{proof}
This essentially follows from the convexity result from \cite[Theorem 11.1]{BBEGZ} which generalizes Berndtsson's convexity result about Ding energy in smooth Fano case \cite{Bern15} to a very general log case. To see this, by using \eqref{eq-KYG} we set $L'=\frac{1}{1+\epsilon} L_\epsilon+\Delta_\epsilon=-K_Y+G$ to be a line bundle on $Y$. Let $p'_1: Y\times\bC\rightarrow \bC$ be the natural projection. Then $e^{-\Phi'}:=e^{-\frac{\Phi_\epsilon}{1+\epsilon}}\frac{1}{|s_{(\Delta_\epsilon)_\bC}|^2}$ is a positively curved (singular) Hermitian metric on $p'^*_1 L'$. Because $K_Y+L'=G=\sum_j \lceil a_j\rceil E''_j$ is exceptional, we see that $H^0(K_Y+L')=\bC\cdot s_{G_\bC}\cong \bC$ and $\bL_{(Y, B_\epsilon)}(\vphi_\epsilon(s))$ is the Bergman kernel of $K_Y+L'$ with respect to the Hermitian metric $e^{-\Phi'}$. 
Moreover, we have 
$$H^1(Y, K_Y+L')=H^1(Y, K_Y+L_\epsilon+\Delta_\epsilon)=0$$ by the Kawamata-Viehweg vanishing theorem.
So all the conditions in \cite[Theorem 11.1]{BBEGZ} are satisfied and, as proved there, Berndtsson's convexity result implies that $\bL_{(Y, B_\epsilon)}(\vphi_\epsilon(s))$ is indeed convex with respect to $s=\log|t|^{-1}$.
\end{proof}

In the following discussion, let $W$ denote the space of $\bC^*$-invariant divisorial valuations on $Y_\bC=Y\times\bC$ such that $w(t)=1$, and $A_{Y_\bC}(w)$ is the log discrepancy of $w$ over $Y_\bC$.
The following theorem can be proved by the similar method as \cite[Theorem 3.1]{BBJ18}. However, since there is a non-effective twisting in our case, we will give the details of proof.
\begin{prop}\label{prop-LBexpan}
Fix $0\le \epsilon\ll 1$.
Let $\Phi_\epsilon=\{\vphi(s)\}_{s\in [0, +\infty)}$ be a psh ray in $\cE^1(Y, L_\epsilon)$ normalized such that $\sup(\vphi(s)-\psi_\epsilon)=0$. With the above notations,
we have the identity:
\begin{equation}\label{eq-LBexpan}
\frac{1}{V_\epsilon(1+\epsilon)}\cdot \lim_{s\rightarrow+\infty} \frac{\bL_{(Y, B_\epsilon)} (\vphi(s))}{s}=
\inf_{w\in W} \left(A_{Y_\bC}(w)-\frac{1}{1+\epsilon}w(\Phi_\epsilon)-w((\Delta_\epsilon)_\bC)+ w(G_\bC)\right)-1,
\end{equation}
where $(\Delta_\epsilon)_\bC=\Delta_\epsilon\times\bC$ and $G_\bC=G\times \bC$.
\end{prop}
\begin{proof}
For simplicity of notations, set $\tilde{\Phi}:=\frac{1}{1+\epsilon}\Phi_\epsilon$. Then $w(\tilde{\Phi})=\frac{1}{1+\epsilon}w(\Phi_\epsilon)$. 
Since the function
\begin{equation}
u_\epsilon(t):=\bL_{(Y, B_\epsilon)}(\vphi(\log|t|^{-1}))=-\log\left(\int_Y e^{-\tilde{\Phi}}\frac{|s_{G}|^2}{|s_{\Delta_\epsilon}|^2}\right)
\end{equation}
is subharmonic on $\bD$ by Lemma \ref{lem-conv}, its Lelong number $\nu$ at the origin coincides with the negative of the left-hand-side of \eqref{eq-LBexpan}.
We need to show that $\nu$ is equal to
\begin{equation}\label{eq-zeta}
\zeta:=\sup_{w\in W} \left(w(\tilde{\Phi})+w((\Delta_\epsilon)_\bC)-w(G_\bC)-A_{Y_\bC}(w)\right)+1.
\end{equation}
By \cite[Proposition 3.8]{Berm15}, $\nu$ is the infimum of all $c\ge 0$ such that:
\begin{equation}
\int_U e^{-(u_\epsilon(t)+(1-c)\log|t|^2)} \sqrt{-1}dt\wedge d\bar{t}=\int_{Y\times U}e^{-(\tilde{\Phi}+(1-c)\log|t|^2)}\frac{|s_{G_\bC}|^2}{|s_{(\Delta_\epsilon)_\bC}|^2} \sqrt{-1}dt\wedge d\bar{t}<+\infty.
\end{equation}
Set $p:=\lfloor c\rfloor$ and $r=c-p\in [0,1)$. Then we have
$$e^{-(\tilde{\Phi}+(1-c)\log|t|^2)}\frac{|s_{G_\bC}|^2}{|s_{(\Delta_\epsilon)_\bC}|^2}
\sqrt{-1}dt\wedge d\bar{t}=|t|^{2p}|s_{G_\bC}|^2 e^{-(\tilde{\Phi}+(1-r)\log|t|+\log|s_{(\Delta_\epsilon)_\bC}|^2)}dV.
$$
It follows from \cite[Theorem B.5]{BBJ18} (see also \cite[Theorem 5.5]{BFJ08}) that
\begin{eqnarray*}
&&\int_{Y\times U}e^{-(\tilde{\Phi}+(1-c)\log|t|^2)}\frac{|s_{G_\bC}|^2}{|s_{(\Delta_\epsilon)_\bC}|^2} \sqrt{-1}dt\wedge d\bar{t}<+\infty\\
&&\hskip 3cm \Longrightarrow\hskip 1cm \sup_{w\in W}\frac{w(\tilde{\Phi})+w((\Delta_\epsilon)_\bC)+(1-r)w(t)}{p w(t)+w(G_\bC)+A_{Y_\bC}(w)}\le 1,
\end{eqnarray*}
where $w$ ranges over all divisorial valuations on $Y_\bC$. By homogeneity and by the $S^1$-invariance of $\Phi$, it suffices to consider $w$ that are $\bC^*$-invariant and normalized by
$w(t)=1$. We then get:
\begin{equation}
w({\Phi})+1\le p+r+w(G_\bC)-w((\Delta_\epsilon)_\bC)+A_{Y_\bC}(w).
\end{equation}
So we get $\zeta \le \nu$.

Conversely, the valuative description of multiplier ideals from \cite[Theorem B.5]{BBJ18} (see also \cite[Theorem 5.5]{BFJ08}) shows that:
\begin{eqnarray}\label{eq-intsuff}
&&\sup_{w\in W} \frac{w(\tilde{\Phi})+w((\Delta_\epsilon)_\bC)+1-r}{p+w(G_\bC)+A_{Y_\bC}(w)}<1\\
&&\hskip 3cm \Longrightarrow \hskip 0.5cm \int_{Y\times U}e^{-(\Phi+(1-c)\log|t|)}\frac{|s_{G_\bC}|^2}{|s_{(\Delta_\epsilon)_\bC}|^2}\sqrt{-1}dt\wedge d\bar{t}<+\infty. \nonumber
\end{eqnarray}
To prove $\zeta\ge \nu$, it suffices to show that for any $\delta>0$ and $a\ge \zeta+\delta$, if we let $\lfloor a\rfloor=p$ and $r=a-p\in [0,1)$, then the following inequality holds:
\begin{eqnarray}\label{eq-sgnu}
\sup_{w\in W}\frac{w(\tilde{\Phi})+w((\Delta_\epsilon)_\bC+1-r}{p+w(G_\bC)+A_{Y_\bC}(w)}<1.
\end{eqnarray}
For any $w\in W$, by \eqref{eq-zeta} we have:
\begin{eqnarray*}
a=p+r \ge \delta+\zeta\ge \delta+w(\tilde{\Phi})+w((\Delta_\epsilon)_\bC)-w(G_\bC)-A_{Y_\bC}(w)+1
\end{eqnarray*}
or equivalently:
\begin{eqnarray*}
p+w(G_\bC)+A_{Y_\bC}(w)\ge  \delta+ w(\tilde{\Phi})+w((\Delta_\epsilon)_\bC)+1-r
\end{eqnarray*}
So we get:
\begin{eqnarray*}
\frac{w(\tilde{\Phi})+w((\Delta_\epsilon)_\bC)+1-r}{p+w(G_\bC)+A_{Y_\bC}(w)}\le 1-\frac{\delta}{p+w(G_\bC)+A_{Y_\bC}(w)}\le 1-\frac{\delta}{p+(1+\lct(G_\bC))A_{Y_\bC}(w)}.
\end{eqnarray*}
On the other hand, because locally $e^{-\tilde{\Phi}}\frac{1}{|s_{(\Delta_\epsilon)_\bC}|^2}\in L^1_{\rm loc}(Y\times\bD^*)$ for $\epsilon\ll 1$, by \cite[Lemma 5.5]{BBJ18}, there exist $\alpha\in (0,1)$ and $C>0$ such that $w(\tilde{\Phi})+w((\Delta_\epsilon)_\bC)\le (1-\alpha)A(w)+C$ for any $w\in W$. So we have
\begin{eqnarray*}
\frac{w(\tilde{\Phi})+w((\Delta_\epsilon)_\bC)+1-r}{p+w(G_\bC)+A_{Y_\bC}(w)}\le \frac{w(\tilde{\Phi})+w((\Delta_\epsilon)_\bC)+1}{A_{Y_\bC}(w)}\le 1-\alpha+\frac{1}{A_{Y_\bC}(w)}.
\end{eqnarray*}
Now it's easy to get the inequality \eqref{eq-sgnu}.
\end{proof}

In the following discussion, let $\Phi$ be the destabilising geodesic ray constructed in section \ref{sec-step1}. 
Then as in section \ref{sec-pertTC}, for any $\epsilon>0$ sufficiently small,
let $\Phi_\epsilon=\Phi+\epsilon \psi_P$ be a psh ray in $\mcE^1(Y, L_\epsilon)$, and $\Phi^\NA_\epsilon$ be the associated non-Archimedean metric (see Definition \ref{defn-PhiNA}). Let $(\mcY, \mcL_{\epsilon,m})$ be the test configuration of $(Y, L_\epsilon)$ constructed in \eqref{eq-defcLepm}-\eqref{eq-defphiepm}), and $\Phi^\NA_{\epsilon,m}\in \mcH^\NA(Y, L_\epsilon)$ be the associated non-Archimedean metric (see Definition \ref{defn-TCNA}).

Set
\begin{eqnarray}\label{eq-LBNA2}
\bL^\NA_{(Y, B_\epsilon)}(\Phi^\NA_\epsilon)&:=&V_\epsilon(1+\epsilon)\cdot \inf_{w\in W}\left(A_{Y_\bC}(w)-\frac{1}{1+\epsilon}w(\Phi_\epsilon)-w((\Delta_\epsilon)_\bC+w(G_\bC)\right)-1\\
&=&V_\epsilon(1+\epsilon)\cdot \inf_{v\in Y^{\rm div}_\bQ}\left(A_{Y_\bC}(v)+\frac{1}{1+\epsilon}\Phi^\NA_\epsilon(v)-v(B_\epsilon)\right).\nonumber
\end{eqnarray}
With Proposition \ref{prop-LBexpan}, the following result can be proved by the similar argument as in \cite{BBJ15}. Again since there is a non-effective twisting in our case, we give the details. We need the following inequalities: for any psh ray $\Phi_\epsilon$ on $L_\epsilon$ and $w\in W$,
\begin{equation}\label{eq-wJPhim}
 w(\cJ(m\Phi_\epsilon))\le  m\; w(\Phi_\epsilon)\le w(\cJ(m\Phi_\epsilon))+ A_{Y_\bC}(w).
\end{equation}
The first inequality holds because $\Phi_{m}$ is less singular than $\Phi$ by Demailly's regularization result.  The second inequality follows from the definition of multiplier ideal $\cJ(m\Phi_\epsilon)$. 
\begin{prop}[see \cite{BBJ15}]
We have the identity:
\begin{equation}\label{eq-limLNAm}
\lim_{m\rightarrow+\infty} \bL_{(Y, B_\epsilon)}^\NA(\Phi^\NA_{\epsilon,m})=\bL^\NA_{(Y, B_\epsilon)}(\Phi^\NA_\epsilon)=\lim_{s\rightarrow+\infty} \frac{\bL_{(Y, B_\epsilon)}(\vphi_\epsilon(s))}{s}.
\end{equation}
\end{prop}
\begin{proof}
The second equality follows from Proposition \ref{prop-LBexpan}.
For simplicity of notations, denote $T'=\frac{1}{(1+\epsilon)V_\epsilon}\bL^\NA_{(Y, B_\epsilon)}(\Phi^\NA_\epsilon)$
and
\[
T_+:=\frac{1}{(1+\epsilon)V_\epsilon}\limsup_{m\rightarrow+\infty} \bL_{(Y, B_\epsilon)}^\NA(\Phi^\NA_{\epsilon,m})\quad \ge  \quad T_-:=\frac{1}{(1+\epsilon)V_\epsilon}\liminf_{m\rightarrow+\infty} \bL_{(Y, B_\epsilon)}^\NA(\Phi^\NA_{\epsilon,m}).
\]
Using \eqref{eq-wJPhim}, we get, for any $\bC^*$-invariant valuation $w$ on $Y_\bC$ with $w(t)=1$,
\begin{eqnarray*}
-\frac{1}{m}w(\cJ(m {\Phi}))&\ge& - w({\Phi}).
\end{eqnarray*}
Adding $A_{Y_\bC}(w)-w((\Delta_\epsilon)_\bC)+w(G_\bC)-1$ on both sides and taking infimum, we get $\frac{1}{(1+\epsilon)V_\epsilon}\bL^\NA_{(Y, B_\epsilon)}(\Phi^\NA_{\epsilon,m})\ge T'$ for any $m$ and hence $T_-\ge T'$. 

On the other hand, for any $\alpha>0$, there exists $w\in W$ such that
$$
A_{Y_\bC}(w)-1-w(\Phi_\epsilon)-w((\Delta_\epsilon)_\bC)+v(G_\bC) \le T'+\alpha.
$$
So we get the inequality:
\begin{eqnarray*}
-\frac{1}{m}w(\cJ(m\Phi_\epsilon))&\le&  \left(-w(\Phi_\epsilon)+\frac{1}{m}A_{Y_\bC}(w)\right)\\
&\le& T'+\alpha-A_{Y_\bC}(w)+1+w((\Delta_\epsilon))-v(G_\bC)+\frac{1}{m}A_{Y_\bC}(w).
\end{eqnarray*}
Taking $\limsup$ as $m\rightarrow+\infty$, we get $T_+\le T'+\alpha$. Since $\alpha>0$ is arbitrary, we get $T_+\le T'$ and hence $T_+=T_-=T'$ as wanted.
\end{proof}

The following key proposition proves the convergence in \eqref{eq-sklimLep}. 
\begin{prop}\label{prop-limLep}
With the above notations, the following convergence holds true:
\begin{equation}\label{eq-limLNAep}
\lim_{\epsilon\rightarrow 0}\bL_{(Y, B_\epsilon)}^\NA(\Phi^\NA_\epsilon)=\bL^\NA(\Phi^\NA).
\end{equation}
\end{prop}
\begin{proof}
Since $w(\Phi+\epsilon \psi_P)=w(\Phi)$, by Proposition \ref{prop-LBexpan}, we have:
\begin{eqnarray*}
\bL^\NA_{(Y, B_\epsilon)}(\Phi^\NA_\epsilon)&=& \inf_{w\in W} \left(A_{Y_\bC}(w)-\frac{1}{1+\epsilon}w(\Phi)-w((\Delta_\epsilon)_\bC)+ w(G_\bC)\right)-1.
\end{eqnarray*}
On the other hand, recall that (see \eqref{eq-LBNA2})
\begin{eqnarray*}
\bL^\NA(\phi)=\inf_{w\in W}\left(A_{(X_\bC, D_\bC)}(w)-w(\Phi)\right)-1.
\end{eqnarray*}

Note that since $A_{(X_\bC, D_\bC)}(w)=A_{Y_\bC}(w)+w(K_{Y_\bC/(X_\bC, D_\bC)})$, we have the following identities:
\begin{eqnarray*}
&&A_{Y_\bC}(w) -\frac{1}{1+\epsilon}w(\Phi)-w((\Delta_\epsilon)_\bC)+w(G_\bC)\\
&=&A_{Y_\bC}(w)-w((\Delta_0)_\bC)+w(G_\bC)-\frac{1}{1+\epsilon}w(\Phi)-\frac{\epsilon}{1+\epsilon} w((E_\theta)_\bC) \\
&=&A_{(X_\bC, D_\bC)}(w)-w(\Phi)+\frac{\epsilon}{1+\epsilon}w(\Phi)-\frac{\epsilon}{1+\epsilon} w((E_\theta)_\bC).
\end{eqnarray*}
This holds for any $\epsilon\ge 0$.
For the simplicity of notations, denote the above equivalent quantity by
\begin{equation}
F_\epsilon(w):=A_{Y_\bC}(w)-w((\Delta_0)_\bC)+w(G_\bC)-\frac{1}{1+\epsilon}w(\Phi)-\frac{\epsilon}{1+\epsilon} w((E_\theta)_\bC).
\end{equation}
Then, by definition, for any $\epsilon\ge 0$, $\bL^\NA_\epsilon(\vphi_\epsilon)= \bI_\epsilon-1$ where 
\begin{eqnarray*}
\bI_{\epsilon}&:=&\inf_{w\in W} F_\epsilon(w). 
\end{eqnarray*}
So we need to prove $\lim_{\epsilon\rightarrow 0}\bI_\epsilon=\bI_0$.

We first show that $\bI_\epsilon$ is uniformly bounded for any $\epsilon\in [0,1)$. To see this note that:
\begin{eqnarray*}
&&A_{Y_\bC}(w)-w((\Delta_0)_\bC)-w(\Phi)-w((E_\theta)_\bC)\\
&&\hskip 2cm \le F_\epsilon(w)=A_{Y_\bC}(w)-w((\Delta_0)_\bC)+w(G_\bC)-\frac{1}{1+\epsilon}w(\Phi)-\frac{\epsilon}{1+\epsilon} w((E_\theta)_\bC)\\
&&\hskip 5cm \le A_{Y_\bC}(w)-w((\Delta_0)_\bC)+w(G_\bC)-\frac{1}{2}w(\Phi).
\end{eqnarray*}
We can assume $\theta_i\ll 1$ such that $e^{-\Phi}\frac{1}{|s_{{\Delta_0}}\cdot s_{E_\theta}|^2}\in L^1_{\rm loc}(Y\times \bD^*)$. By \cite[Lemma 5.5]{BBJ18}, there exist $\tau\in (0,1)$ and a constant $C_1>0$ such that:
\begin{eqnarray}\label{eq-openness}
w((\Delta_0)_\bC)+w(\Phi)+w((E_\theta)_\bC))\le (1-\tau) A(w)+C_1 \quad \text{ for all } w\in W
\end{eqnarray}
So we easily get that there exists a constant $C>0\in \bR$ independent of $\epsilon\in [0,1)$ such that
\begin{eqnarray*}
|\bI_\epsilon|\le C.
\end{eqnarray*}
So we get $\bI_\epsilon=\inf_{F_\epsilon(w)\le C+1} F_\epsilon(w)$. 
Pick any $w\in W$ such that
\begin{eqnarray*}
F_\epsilon(w)\le A_{Y_\bC}(w)-w((\Delta_0)_\bC)+w(G_\bC)-\frac{1}{1+\epsilon}w(\Phi)-\frac{\epsilon}{1+\epsilon} w((E_\theta)_\bC)<C+1.
\end{eqnarray*}
We can estimate, by using \eqref{eq-openness}, that
\begin{eqnarray*}
A_{Y_\bC}(w)&\le& C+1+w((\Delta_0)_\bC)-w(G_\bC)+\frac{1}{1+\epsilon} w(\Phi)+\frac{\epsilon}{1+\epsilon} w((E_\theta)_\bC)\\
&\le&C+1+w((\Delta_0)_\bC)+w(\Phi)+w((E_\theta)_\bC)\\
&\le&C+1+C_1+(1-\tau)A(w).
\end{eqnarray*}
This implies $A_{Y_\bC}(w)\le \frac{C'}{\tau}$ with $C'=C+1+C_1$. If we denote $W'=\{w\in W; A_{Y_\bC}(w)\le \frac{C'}{\tau}\}$, then we get:
\begin{equation}\label{eq-bIW'}
\bI_\epsilon=\inf_{w\in W'}F_\epsilon(w).
\end{equation}
For any $w\in W'$, we have, by using \eqref{eq-openness} again, that:
\begin{eqnarray*}
F_\epsilon(w) &=& A_{Y_\bC}(w)-w((\Delta_0)_\bC)+w(G_\bC)-w(\Phi)+\frac{\epsilon}{1+\epsilon}w(\Phi)-\frac{\epsilon}{1+\epsilon} w((E_\theta)_\bC)\\
&\le& F_0(w)+\frac{\epsilon}{1+\epsilon}w(\Phi) \le F_0(w)+\frac{\epsilon}{1+\epsilon}\left((1-\tau)A(w)+C_1\right)\\
&\le& F_0(w)+\frac{\epsilon}{1+\epsilon}\left((1-\tau)\frac{C'}{\tau}+C_1\right).
\end{eqnarray*}
and also:
\begin{eqnarray*}
F_\epsilon(w)&=&A_{Y_\bC}(w)-w((\Delta_0)_\bC)+w(G_\bC)-w(\Phi)+\frac{\epsilon}{1+\epsilon}w(\Phi)-\frac{\epsilon}{1+\epsilon} w((E_\theta)_\bC)\\
&\ge& A_{Y_\bC}(w)-w((\Delta_0)_\bC)+w(G_\bC)-w(\Phi)-\frac{\epsilon}{1+\epsilon}w((E_\theta)_\bC)\\
&\ge& F_0(w)-\frac{\epsilon}{1+\epsilon}((1-\tau)A(w)+C_1)\\
&\ge& F_0(w)-\frac{\epsilon}{1+\epsilon}((1-\tau)\frac{C'}{\tau}+C_1).
\end{eqnarray*}
Letting $C''=(1-\tau)\frac{C'}{\tau}+C_1$ and taking infimum, we get:
\begin{eqnarray*}
\inf_{w\in W'}F_\epsilon(w)-C''\frac{\epsilon}{1+\epsilon} \le \inf_{w\in W'}F_0(w)\le \inf_{w\in W'}F_\epsilon(w)+\frac{\epsilon}{1+\epsilon}C''.
\end{eqnarray*}
Now \eqref{eq-limLNAep} follows by letting $\epsilon\rightarrow 0$ and using \eqref{eq-bIW'}.

\end{proof}

\subsection{Step 4: Uniform Ding-stability of $(Y, B_\epsilon)$}\label{sec-uniDing}
Recall that we have the following identity from \eqref{eq-KYG}
\begin{equation}
-(K_Y+B_\epsilon)=\frac{1}{1+\epsilon}(\mu^*(-K_X-D)+\epsilon P)=\frac{1}{1+\epsilon}L_\epsilon. 
\end{equation}
where
$$B_\epsilon=D'+\sum_k b_k E_k+\frac{\epsilon}{1+\epsilon}\sum_k \theta_k E_k=:B_0+\frac{\epsilon}{1+\epsilon}E_\theta.$$
The following result is analogous to \cite[Proposition 3.1]{LTW17}, which is based on the valuative criterion from Theorem \ref{thm-MvsD}. 
\begin{prop}\label{prop-DepNAproper}
With the above notations, assume that $(X, -K_X-D)$ is uniformly Ding stable with $\delta(X, D)=\delta_0>1$. Then
there exists a constant $C>0$ and a constant $\epsilon_*>0$ such that 
for any $0<\epsilon\ll \epsilon_*$, we have the following identity on $\cH^{\NA}(Y, L_\epsilon)$:
\begin{eqnarray*}
\bfD_{(Y, B_\epsilon)}^\NA=-\bfE_{L_\epsilon}^\NA+\bL_{(Y, B_\epsilon)}^\NA\ge  \left(1-\left((1-C\epsilon)\delta_{0}\right)^{-1/n}\right) \bfJ_{L_\epsilon}^\NA.
\end{eqnarray*}
\end{prop}

\begin{proof}
By Theorem \ref{thm-MvsD}, we just need to show that $\delta(Y, B_\epsilon)\ge (1-C\epsilon)\delta_0$.
Consider the quantity:
\begin{equation}
\Theta(\epsilon):=\frac{A_{(Y, B_\epsilon)}(E)(-K_Y-B_\epsilon)^n}{\int_0^\infty \vol_Y(-K_Y-B_\epsilon-xE)dx}.
\end{equation}
Then $\delta(Y, B_\epsilon):=\inf_E \Theta(\epsilon)$. Moreover, by the definition of $\delta(X, D)$ in \eqref{eq-deltaXY}, we have:
\begin{eqnarray*}
\Theta(0)=\frac{A_{(Y, B_0)}(E)(-K_Y-B_0)^n}{\int_0^\infty \vol_Y(-K_Y-B_0-xE)}dx=\frac{A_X(E)(-K_X-D)^n}{\int_0^\infty\vol(-K_X-D-x E)dx}\ge \delta(X,D)=\delta_{0}.
\end{eqnarray*}
So it is enough to prove that $\Theta(\epsilon)\ge (1-C\epsilon)\Theta(0)$. Consider the ratio:

\begin{eqnarray*}
{\bf R}(\epsilon):=\frac{\Theta(\epsilon)}{\Theta(0)}&=&\frac{A_{(Y, B_\epsilon)}(E)}{A_{(Y, B_0)}(E)}\cdot \frac{\int_0^{+\infty}\vol_Y(-K_Y-B_0-xE)dx}{\int_0^\infty \vol_Y(-K_Y-B_\epsilon-xE)dx}\cdot \frac{(-K_Y-B_\epsilon)^n}{(-K_Y-B_0)^n}\\
&=&R_1\cdot R_2\cdot R_3.
\end{eqnarray*}
The second ratio $R_2\ge 1$ because $-B_\epsilon=-B_0-\frac{\epsilon}{1+\epsilon} E_\theta\le -B_0$ and volume function is increasing along effective divisors. The factor $R_3$, which does not depend on $E$, clearly goes to $1$ as $\epsilon\rightarrow 0$. To estimate $R_1$, we use the decomposition $B_0=\Delta_0-G$ with $\lfloor \Delta_0\rfloor=0$ (see \eqref{eq-KYG}) and estimate as follows:
\begin{eqnarray*}
R_1&=&\frac{A_{(Y, B_\epsilon)}(E)}{A_{(Y, B_0)}(E)}=\frac{A_{(Y, B_0)}(E)-\frac{\epsilon}{1+\epsilon}\ord_E(E_\theta)}{A_{(Y, B_0)}(E)}\\
&=&1-\frac{\epsilon}{1+\epsilon}\frac{\ord_E(E_\theta)}{A_Y(E)-\ord_E(\Delta_0)+\ord_E(G)}\\
&\ge& 1-\frac{\epsilon}{1+\epsilon}\frac{\ord_E(E_\theta)}{A_Y(E)-\ord_E(\Delta_0)}\\
&\ge& 1-\frac{\epsilon}{1+\epsilon}(\lct(Y, \Delta_0; E_\theta))^{-1}.
\end{eqnarray*}
So ${\bf R}(\epsilon)\ge 1-C\epsilon$ for some $C>0$ independent of $E$. This concludes the proof.
\end{proof}

\subsection{Step 5: Completion of the proof}\label{sec-step5}

With the above notations and preparations, we can complete the proof of our main result. On the one hand,
by \eqref{eq-EPhi*}-\eqref{eq-Dvphiupb}, 
\begin{equation}\label{eq-bLsmall}
\bL^\NA(\Phi^\NA)=\lim_{s\rightarrow+\infty} \frac{\bL(\vphi(s))}{s}
=
\lim_{s\rightarrow+\infty} \frac{\bfD(\vphi(s))}{s}+\lim_{s\rightarrow+\infty}\frac{\bfE(\vphi(s))}{s}\le \gamma-1.
\end{equation}

On the other hand, by Proposition \ref{prop-DepNAproper}, we have, with $\tilde{\delta}_\epsilon=1-((1-C\epsilon)\delta_0)^{-1/n}$,
\begin{eqnarray*}
\bL_{(Y, B_\epsilon)}^\NA(\Phi^\NA_{\epsilon,m})&=&\bfD^\NA_{(Y, B_\epsilon)}(\Phi^\NA_{\epsilon,m})+\bfE_{L_\epsilon}^\NA(\Phi^\NA_{\epsilon,m}) \ge \tilde{\delta}_\epsilon \bfJ_{L_\epsilon}^\NA(\Phi^\NA_{\epsilon,m})+\bfE_{L_\epsilon}^\NA(\Phi^\NA_{\epsilon,m})\\
&=&(1-\tilde{\delta}_\epsilon)\bfE_{L_\epsilon}^\NA(\Phi^\NA_{\epsilon,m})=((1-C\epsilon)\delta_0)^{-1/n} \bfE_{L_\epsilon}^\NA(\Phi^\NA_{\epsilon,m})\\
&\ge& ((1-C\epsilon)\delta_0)^{-1/n}\bfE'^\infty_{\psi_\epsilon}(\Phi_\epsilon).
\end{eqnarray*}
The second equality uses $\sup \Phi^\NA_{\epsilon,m}=0$ (see \eqref{eq-phimcI}).
The last inequality uses \eqref{eq-ENAepmlb}. 
Taking $m\rightarrow+\infty$ and using \eqref{eq-limLNAm}, we get the inequality:
\begin{eqnarray*}
\bL^\NA_{(Y, B_\epsilon)}(\Phi^\NA_\epsilon)\ge ((1-C\epsilon)\delta_0)^{-1/n}\bfE'^\infty_{\psi_\epsilon}(\Phi_\epsilon).
\end{eqnarray*}
Now we let $\epsilon\rightarrow 0$ and use Proposition \ref{prop-limLep} and Proposition \ref{prop-E'infconv} to get:
\begin{eqnarray*}
\bL^\NA(\Phi^\NA)\ge \delta_0^{-1/n} E'^\infty(\Phi)=-\delta_0^{-1/n}>-1.
\end{eqnarray*}

But this contradicts \eqref{eq-bLsmall} when $0<\gamma<1-\delta_0^{-1/n}$. So the proof is completed.

\vskip 3mm

\noindent
Department of Mathematics, Purdue University, West Lafayette, IN 47907-2067.\\
\noindent
{\it Current Address:} 
Department of Mathematics, Rutgers University,  Piscataway, NJ 08854-8019.

\noindent
{\it E-mail address:} chi.li@rutgers.edu

\vskip 2mm

\noindent
School of Mathematical Sciences and BICMR, Peking University, Yiheyuan Road 5, Beijing, P.R.China, 100871

\noindent {\it E-mail address:} tian@math.princeton.edu

\vskip 2mm

\noindent
School of Mathematical Sciences, Zhejiang University, Zheda Road 38, Hangzhou, Zhejiang,
310027, P.R. China

\noindent {\it E-mail address:} wfmath@zju.edu.cn

\end{document}